\documentclass[10pt,a4paper]{article}

\usepackage[utf8]{inputenc}
\usepackage[english]{babel}
\usepackage{amsmath,amsfonts,amssymb,amsthm}
\usepackage{graphicx}
\usepackage{mathtools} 

\newtheorem{prop}{Proposition}
\newtheorem{thm}{Theorem}

\theoremstyle{definition}
\newtheorem{de}{Definition}[]
\newtheorem{ex}{Example}

\renewcommand{\textbf}[1]{\begingroup\bfseries\mathversion{bold}#1\endgroup}

\def\cE{\mathcal{E}}

\def\cM{\mathcal{M}}
\def\oM{\overline{\mathcal{M}}}
\def\hM{{\widehat{M}}}
\def\N{{\mathbb{N}}}

\def\Q{{\mathbb{Q}}}

\def\tv{{\tilde{v}}}

\def\cZ{{\mathcal{Z}}}

\begin{document}



\begin{center}
\Large{\textbf{Complete descriptions of the tautological rings of the moduli spaces of curves of genus lower or equal to 4 with marked points}}

\end{center}

~\\

\begin{center}
MALICK CAMARA
\end{center}

~\\

\begin{center}
\textbf{Abstract:} We study here the tautological rings of the moduli spaces of compact Riemann surfaces of genus $1,2,3$ and $4$ with marked points. The paper presents the complete descriptions of these rings by describing the groups of all degrees.   
\end{center}


\tableofcontents

\section{Introduction}
\subsection{Tautological rings of moduli spaces}
The topic of this 	article is the cohomology ring of the moduli space $\cM_{g,n}$ of genus $g$ smooth curves with $n$ marked points, more precisely, its subring called the {\em tautological ring} and denoted~$R^*(\cM_{g,n}) \subset H^{\rm even}(\cM_{g,n},\Q)$. 

The study of the tautological ring has been initiated by \linebreak D.~Mumford in~\cite{Mum}. In 1999, C.~Faber computed the tautological rings of the moduli space $\cM_g$ of smooth curves without marked points for $g \leq 15$~\cite{Fab}. Building on this work, Faber and Zagier conjectured a set of relations between the tautological classes of $\cM_g$, namely the $\kappa$-classes; these were later proven to be true relations by R.~Pandharipande and A.~Pixton~\cite{PP} using localization on the space of stable quotients.
 
A generalization of these relations to the space $\oM_{g,n}$ of stable curves was proposed by A. Pixton in \cite{Pix}. These conjectural relations were proven to be true relations by R.~Pandharipande, A.~Pixton, and D.~Zvonkine in~\cite{PPZ}. The proof is based on the study of a cohomological field theory obtained from Witten's $r$-spin classes (see \cite{Wit}, \cite{PV}, \cite{Chi}, \cite{Moc}, \cite{FJR} for details about these classes). Relations on $\cM_{g,n}$ can be obtained by restriction of these relations from $\oM_{g,n}$. 

The tautological ring of $\cM_{g,n}$ is generated by tautological classes $\psi_1, \dots, \psi_n \in R^1$ associated with the marked points and by the classes $\kappa_m \in R^m$ for $m \geq 1$ 
. Madsen and Weiss~\cite{MadWei} proved that there are no relations between these classes in degree $d \leq g/3$. Moreover, for an integer $d$ satisfying $1 \leq d \leq g/3$, we have $H^{2d} (\cM_{g,n},\Q) = R^d(\cM_{g,n})$ and $H^{2d-1}(\cM_{g,n}, \Q)=0$. 

Building on E.~Looijenga's work~\cite{Loo}, E.~Ionel~\cite{Ion1} proved the following vanishing property: $R^d = 0$ for $d \geq g \geq 1$. The rank of the group $R^{g-1} (\cM_{g,n})$ is known to be equal to~$n$ for any~$g \geq 2$. Moreover, the classes $\psi_i^{g-1}$ form a basis of this space, and an explicit expression in this basis of all other elements is known. Following a conjecture by D.~Zvonkine this was proved by A.~Buryak, S.~Shadrin, and D.~Zvonkine in~\cite{BSZ}.

\subsection{Main results}

\paragraph{A complete description of the rings $R^*(\cM_{g,n})$ for \textbf{$g \leq 4$.}}

An element of the tautological ring can be shown to vanish by proving that it lies in the span of Pixton's tautological relations restricted to $\cM_{g,n}$. On the other hand, since the top degree (that is, degree $g-1$) part of the tautological ring is explicitly known by~\cite{BSZ}, one can show that an element of $R^d(\cM_{g,n})$ does not vanish by proving that its product with a well chosen element of $R^{g-1-d}(\cM_{g,n})$ does not vanish in $R^{g-1}(\cM_{g,n})$. Thus the tautological ring is in a way ``bounded'' from above and from below. It turns out that for $g \leq 4$ these bounds coincide with each other; in other words, every element that cannot be shown to vanish can be shown to not vanish and vice versa. 

Proving this eventually boils down to computing ranks of large matrices. For instance, in the hardest case, $g=4$, $d=2$, one has to compute the rank of a matrix of size
$$
\left(\frac{n(n+1)}2+1\right)  \times  (n^2+1).
$$
This involves the following steps. One needs to extract a well-chosen square matrix out of the rectangular matrix above;  find a family of eigenvectors of this matrix (treated as the matrix of a linear map); find a family of invariant planes of the matrix; show that the determinant on each invariant plane does not vanish; determine the span of the eigenvectors and invariant planes; determine the action of the linear map in the quotient space by this span; show that the determinant of the linear map on the quotient space does not vanish; treat differently the cases where, due to a numerical exception, one of the invariant planes contains one of the eigenvectors. It would be interesting to find a geometric interpretation of all these eigenvectors and invariant planes, but so far we do not have any.

Denote by $r^d_g(n)$ the rank of $R^d(\cM_{g,n})$.

\begin{thm}
We have
$$
\begin{array}{llll}
r^0_1(n) = 1, \\
r^0_2(n) = 1, & r^1_2(n) = n, \\
r^0_3(n) = 1, & r^1_3(n) = n+1, & r^2_3(n) = n, \\
r^0_4(n) = 1, & r^1_4(n) = n+1, & r^2_4(n) = \frac{n(n+1)}2+1, & r^3_4(n) = n.
\end{array}
$$
\end{thm}

As a by-product of our computations we also show that in genus 3 and degree~2, the relations of Buryak-Shadrin-Zvonkine are obtained from Pixton's relations by a change of variables whose coefficients are polynomials in~$n$.

\section{Preliminaries}

\subsection{The tautological ring}

In this section we briefly introduce some of the basic definitions involved in the study of the tautological ring of the moduli spaces of curves. We refer to \cite{Zvo} for more details.\\

Writing $\oM_{g,n}$ the moduli space of stable curves of genus $g$ with $n$ marked points, we will always assume that $g$ and $n$ satisfy the stability condition
$$2g-2+n>0.$$

Let's write $\pi:\oM_{g,n+1}\rightarrow\oM_{g,n}$ the forgetful map which forgets the last marked point. It will be convenient to write $\pi$ maps forgetting several points as well. The context will make it clear. \\

We write $\psi_{i}\in H^2(\oM_{g,n},\Q)$, for $1\leq i\leq n$, the first Chern class of the line bundle over $\oM_{g,n}$ whose fiber over any point is the cotangent line at the $i$-th marked point of the corresponding curve.\\

We define the $\kappa$-classes as follows. For $k\in\N$, 
$$\kappa_{k}:=\pi_{*}(\psi_{n+1}^{k+1}).$$
For $m$ positive integers $k_1,\dots, k_{m}$, with $m\geq 1$, we write 
$$\kappa_{k_{1},\dots,k_{m}}:=\pi_{*}(\psi_{n+1}^{k_{1}+1}\psi_{n+2}^{k_{2}+1}\dots\psi_{n+m}^{k_{m}+1})$$
These classes can be expressed as polynomials of $\kappa$-classes with one index.\\

The $k$-th Chern class of the Hodge bundle will be denoted $\lambda_{k}\in H^{2k}(\oM_{g,n},\Q)$.

\begin{de}
The {\em tautological ring} $R^*(\cM_{g,n})$ is the subring of $H^*(\cM_{g,n}, \Q)$ generated by the classes $\psi_1, \dots, \psi_n$ and $\kappa_1, \kappa_2, \dots$. 
\end{de}

The study of this ring boils down to describing the relations between the $\psi$- and $\kappa$-classes.  In this section we describe the following results.
\begin{itemize}
\item Pixton's relations (more precisely, their restrictions from $\oM_{g,n}$ to $\cM_{g,n}$): this is a family of relations among the $\psi$- and $\kappa$-classes that is conjectured to be complete. 
\item The vanishing, socle, and top intersection properties of the tautological rings: these properties completely describe the tautological ring in degrees $d \geq g-1$.
\item Mumford's stability: this is a claim that the tautological ring stabilizes to a free ring as $g \to \infty$. 
\item Mumford's formula for the $\lambda$-classes: this formula expresses the $\lambda$-classes in terms of $\psi$- and $\kappa$-classes.
\end{itemize}

\subsection{Pixton's relations}

In \cite{Pix}, A. Pixton conjectured a set of relations for the tautological ring of moduli spaces of stable nodal curves with marked points generalizing Faber's relations. These relations are proven to be true relations in \cite{PPZ} and they play a crucial role in the work presented here. We present the restriction of these relations to the moduli space of smooth stable curves with marked points.

We have the series
$$A:=\sum_{n\geq0}\frac{(6n)!}{(2n)!(3n)!}T^n=1+60T+27720T^2...$$
and
$$B:=\sum_{n\geq0}\frac{6n+1}{6n-1}\frac{(6n)!}{(2n)!(3n)!}T^n=-1+84T+32760T^2...$$

For any natural number $i$ not congruent to $2$ modulo 3, we write
$$C_{3i}=T^iA$$
$$C_{3i+1}=T^iB.$$

For a power series $S$, $[S]_{T^n}$ denotes the $n$th coefficient. We transform the power series $S$ variables $T^n$ into a power series denoted $\{S\}$ in the variables $K_nT^n$, so we have

$$\lbrace S\rbrace:=\sum_{n\geq0}[S]_{T^n}K_nT^n.$$

Let $l$ and $e_1,\dots,e_l$  be a nonnegative integer and let $\kappa$ denote the linear operator defined by
$$\kappa(K_{e_1}...K_{e_l})=\kappa_{e_1,...,e_l}=\sum_{\tau\in S_l}\prod_{c \mbox{ cycle in } \tau}\kappa_{e_c},$$
with $e_c$ the sum of the $e_j$ appearing in the cycle~$c$.

Let $\sigma$ be a partition with no part congruent to 2 modulo 3 and $a_1,...,a_n$ positive integers not congruent to 2 modulo 3.

If 
$$\left\{
\begin{array}{cl}
3d\geq g+1+\sum\sigma_i+\sum a_i\\
3d\equiv g+1+\sum\sigma_i+\sum a_i\mod 2
    \end{array}
\right.,$$
then 
$$[\kappa(\exp(\lbrace 1-A\rbrace)\lbrace C_{\sigma_1}\rbrace...\lbrace C_{\sigma_l}\rbrace)\prod C_{a_i}(\psi_iT)]_{T^2}=0.$$

\subsection{Vanishing, socle, and top intersection for with no marked point} 

These properties were conjectured by C.~Faber~\cite{Fab} and later proved by E.~Looijenga~\cite{Loo}, E.~Getzler and R.~Pandharipande~\cite{GP}. Here we assume that $g \geq 2$.

\paragraph{Vanishing.} We have $R^d(\cM_g)= 0$ for $d \geq g-1$.

\paragraph{Socle.} The rank of $R^{g-2}(\cM_g)$ is equal to~1 and $R^{g-2}(\cM_g)$ is spanned by $\kappa_{g-2}$.

\paragraph{Top intersection.} For positive integers $k_1,\dots ,k_m$ such that\linebreak $\sum_{i=1}^m k_i=g-2$, we have
$$
\kappa_{k_1,\dots ,k_m}=\frac{(2g-3+m)!(2g-1)!!}{(2g-1)!\prod_{i=1}^m(2k_i+1)!!}\kappa_{g-2} \in R^{g-2}(\cM_{g-2}).
$$

\subsection{Vanishing, socle, and top intersection with marked points,\\
Buryak-Shadrin-Zvonkine Relations}

~~~~These properties were proved by E.~Ionel~\cite{Ion1}, A.~Buryak, S.~Shadrin, and D.~Zvonkine~\cite{BSZ}. The tautological ring of $\cM_{1,n}$ is isomorphic to~$\Q$; thus we can assume that $g \geq 2$.

\paragraph{Vanishing.} We have $R^d(\cM_{g,n})= 0$ for $d \geq g$.

\paragraph{Socle.} The rank of $R^{g-1}(\cM_{g,n})$ is equal to~$n$ and $R^{g-1}(\cM_{g,n})$ is spanned by $\psi_1^{g-1}, \dots, \psi_n^{g-1}$.

\paragraph{Top intersection.} For nonnegative integers $d_1, \dots, d_n$ and positive integers $k_1,\dots ,k_m$ such that $\sum_{i=1}^n\psi_i+\sum_{i=1}^m k_i=g-1$, we have
\begin{align*}
\prod_{i=1}^n\psi_i^{d_i}\, \kappa_{k_1,...,k_m}=&\frac{(2g-1)!!}{\prod(2d_i+1)!!\prod(2k_j+1)!!}\frac{(2g-3+n+m)!}{(2g-2+n)!}\\
&~~~~~~~~~~~\times
\sum_{i=1}^n \frac{(2g-2+n) d_i + \sum k_j}{g-1} \psi_i^{g-1}. 
\end{align*}

\begin{ex} 
For $g=3$ we have
\begin{align*}
\kappa_2&=\sum_{i=1}^n\psi_i^2,\\
\kappa_{1,1}&=\frac{5}{3}(n+5)\sum_{i=1}^n\psi_i^2,\\
\kappa_1\psi_i&=\frac{5}{6}(n+5)\psi_i^2+ \frac56\sum_{j\neq i}\psi_j^2,\\
\psi_i\psi_j&=\frac{5}{6}(\psi_i^2 + \psi_j^2).
\end{align*}
\end{ex}

\section{The tautological rings in genus 1,2,3} 

We study here the tautological rings of moduli spaces of smooth Riemann surfaces of genus $g \leq 3$ with $n$ marked points.

\begin{prop} \ \\
{\bf Genus 1.} We have
\begin{align*}
R^0(\cM_{1,n}) &= \Q,\\
R^d(\cM_{1,n}) &= 0 \quad \mbox{ for } d \geq 1.
\end{align*}
{\bf Genus 2.} We have
\begin{align*}
R^0(\cM_{2,n}) &= \Q,\\
R^1(\cM_{2,n}) &= \Q \langle \psi_1, \dots, \psi_n \rangle,\\
R^d(\cM_{2,n}) &= 0 \quad \mbox{ for } d \geq 2.
\end{align*}
The only relation of degree~1 is $\kappa_1 = \sum_{i=1}^n \psi_i$.\\
{\bf Genus 3.} We have
\begin{align*}
R^0(\cM_{3,n}) &= \Q,\\
R^1(\cM_{3,n}) &= \Q \langle \kappa_1, \psi_1, \dots, \psi_n \rangle,\\
R^2(\cM_{3,n}) &= \Q \langle \psi_1^2, \dots, \psi_n^2 \rangle,\\
R^d(\cM_{3,n}) &= 0 \quad \mbox{ for } d \geq 3.
\end{align*}
The relations in degree~2 are listed in Example 1.
\end{prop}

\begin{proof}
The vanishing for $d \geq g$ is known by~\cite{Ion1}. The rank and the relations in $R^{g-1}(\cM_{g,n})$ are known by~\cite{BSZ}. The group $R^0$ is always equal to~$\Q$. The only remaining case is $R^1(\cM_{3,n})$ where, by Mumford's stability~\cite{MadWei} there are no relations. 
\end{proof}

\subsection{The top tautological group in genus 3}
We use our two  families of relations in $R^2(\cM_{3,n})$. One is due to Buryak-Shadrin-Zvonkine~\cite{BSZ}
. This family is known to be complete and is used to determine the rank of $R^2(\cM_{3,n})$.  The other family is that of Pixton's relations \cite{Pix}
. The completeness of this family is, in general, not known. We prove it here for $g=3$, degree~2 and any~$n$.

\begin{thm}
Pixton's relations give a complete system of relations in $R^2(\cM_{3,n})$ for all~$n$. Buryak-Shadrin-Zvonkine's relations are linear combinations of Pixton's relations whose coefficients are polynomials in~$n$.
\end{thm}

\begin{proof} Pixton's relations are labeled by $n$-tuples $a_1,\dots,a_n$ and a partition $\sigma$ satisfying the conditions $3d \geq g+1+\sum^n_{i=1}a_i+|\sigma|$ and $g+1+\sum_{i=1}^n a_i + \sum \sigma_i\equiv 0 \pmod 2$
. In our case, these conditions become
$$\left\{
\begin{array}{cl}
\sum_{i=1}^n a_i+\sum\sigma_i\leq 2\\
\sum_{i=1}^n a_i+\sum\sigma_i\equiv 0~~\mod 2
    \end{array}
\right.$$
This implies that there are 5 possibilities listed below.

\begin{enumerate}
\item all the $a_i$ are zero and $\sigma=\emptyset$;
\item $a_k=2$ for some $k$, the others $a_i$ are zero and $\sigma=\emptyset$;
\item all the $a_i$ are zero and $\sigma=\lbrace 2\rbrace$;
\item $a_k=1$ for some $k$, the other $a_i$ are zero and $\sigma=\lbrace 1\rbrace $;
\item $a_k=a_l=1$ for some $k$ and some $l$, the others $a_i$ are zero and $\sigma=\emptyset$.
\end{enumerate}

The corresponding Pixton's relations for these cases are:

$$
\begin{array}{ll}
(1) &
\displaystyle 35\sum_{i=1}^n\psi_i^2 + 6\sum_{i<j}\psi_i\psi_j - 6\sum_{i=1}^n\kappa_1\psi_i - 35\kappa_2 + 3\kappa_{1,1}=0,\\[2em]
(2)_k &
\displaystyle 35\sum_{i=1}^n\psi_i^2 -45\psi_k^2- 10\sum_{i=1,i\neq k}^n\psi_k\psi_i + 6\sum_{i<j}\psi_i\psi_j + 10\kappa_1\psi_k \\ & \displaystyle -6\sum_{i=1,i\neq k}^n\kappa_1\psi_i -35\kappa_2 + 3\kappa_{1,1}=0,\\[2em]
(3) &
\displaystyle 35\kappa_0\sum_{i=1}^n\psi_i^2 + 6\kappa_0\sum_{i<j}\psi_i\psi_j -10\sum_{i=1}^n\kappa_1\psi_i -6\sum_{i=1}^n\kappa_{0,1}\psi_i \\
 & -45\kappa_2 -35\kappa_{0,2}+ 10\kappa_{1,1}+ 3\kappa_{0,1,1}=0,\\[2em]
(4)_k &
\displaystyle 35\kappa_0\sum_{i=1,i\neq k}^n \psi_i^2 + 6\kappa_0\sum_{i,j\neq k}\psi_i\psi_j - 6\sum_{i=1,\neq k}^n\kappa_{0,1}\psi_i \\
& - 35\kappa_{0,2} + 3\kappa_{0,1,1}=0,\\[2em]
(5)_{k,l} &
\displaystyle 35\sum_{i=1,i\neq k,l}^n \psi_i^2 + 6\sum_{i,j\neq k,l}\psi_i\psi_j - 6\sum_{i=1,i\neq k,l}^n\kappa_1\psi_i\\
& - 35\kappa_2 + 3\kappa_{1,1}=0.
\end{array}
$$

Eliminating the zeroes in the indices of the $\kappa$-classes we get
$$
\begin{array}{ll}
(1) &
\displaystyle 35\sum_{i=1}^n\psi_i^2 + 6\sum_{i<j}\psi_i\psi_j - 6\sum_{i=1}^n\kappa_1\psi_i - 35\kappa_2 + 3\kappa_{1,1}=0,\\[2em]
(2)_k &
\displaystyle 35\sum_{i=1,i\neq k}^n\psi_i^2-45\psi_k^2-10\sum_{i=1,i\neq k}^n\psi_k\psi_i+6\sum_{i<j,i,j\neq k}\psi_i\psi_j\\
&+10\kappa_1\psi_k\displaystyle -6\sum_{i=1,i\neq k}^n\kappa_1\psi_i
-35\kappa_2+3\kappa_{1,1}=0,\\[2em]
(3) &
\displaystyle 35\kappa_0\sum_{i=1}^n\psi_i^2+ 6\kappa_0\sum_{i<j}\psi_i\psi_j -(16+6\kappa_0)\sum_{i=1}^n\kappa_1\psi_i \\
& \displaystyle -(80+35\kappa_0)\kappa_2+(16+3\kappa_0)\kappa_{1,1}=0,\\[2em]
(4)_k &
\displaystyle 35\kappa_0\sum_{i=1,i\neq k}^n \psi_i^2 +6\kappa_0\sum_{i,j\neq k}\psi_i\psi_j - 6(1+\kappa_0)\sum_{i=1,i\neq k}^n\kappa_1\psi_i  \\
& \displaystyle -35(1+\kappa_0)\kappa_2+ 3(2+\kappa_0)\kappa_{1,1}=0,\\[2em]
(5)_{k,l} &
\displaystyle 35\sum_{i=1,i\neq k,l}^n \psi_i^2 + 6\sum_{i,j\neq k,l}\psi_i\psi_j - 6\sum_{i=1,i\neq k,l}^n\kappa_1\psi_i\\
& - 35\kappa_2 + 3\kappa_{1,1}=0,
\end{array}
$$
or,  taking into account that $\kappa_0 = 2g-2+n = n+4$, 
$$
\begin{array}{ll}
(1) &
\displaystyle 35\sum_{i=1}^n\psi_i^2 + 6\sum_{i<j}\psi_i\psi_j - 6\sum_{i=1}^n\kappa_1\psi_i - 35\kappa_2 + 3\kappa_{1,1}=0,\\[2em]
(2)_k &
\displaystyle 35\sum_{i=1,i\neq k}^n\psi_i^2-45\psi_k^2-10\sum_{i=1,i\neq k}^n\psi_k\psi_i+6\sum_{i<j,i,j\neq k}\psi_i\psi_j\\
&+10\kappa_1\psi_k\displaystyle -6\sum_{i=1,i\neq k}^n\kappa_1\psi_i
-35\kappa_2+3\kappa_{1,1}=0,\\[2em]
(3) &
\displaystyle 35(n+4)\sum_{i=1}^n\psi_i^2+ 6(n+4)\sum_{i<j}\psi_i\psi_j -(6n+40)\sum_{i=1}^n\kappa_1\psi_i \\
& \displaystyle -(35n+220)\kappa_2+(3n+28)\kappa_{1,1}=0,\\[2em]
(4)_k &
\displaystyle 35(n+4)\sum_{i=1,i\neq k}^n \psi_i^2 +6(n+4)\sum_{i,j\neq k}\psi_i\psi_j\\
& - 6(n+5)\sum_{i=1,i\neq k}^n\kappa_1\psi_i \displaystyle -35(n+5)\kappa_2+ 3(n+6)\kappa_{1,1}=0,\\[2em]
(5)_{k,l} &
\displaystyle 35\sum_{i=1,i\neq k,l}^n \psi_i^2 + 6\sum_{i,j\neq k,l}\psi_i\psi_j - 6\sum_{i=1,i\neq k,l}^n\kappa_1\psi_i\\
& - 35\kappa_2 + 3\kappa_{1,1}=0.
\end{array}
$$

We keep the same definitions of the indices $k$ and $l$, for $g=3$, there are four different types of Buryak-Shadrin-Zvonkine's relations (see Example 1):

$$
\begin{array}{ll}
(a)_{k,l} & \displaystyle \psi_k\psi_l=\frac{5}{6}(\psi_k^2 + \psi_l^2),\\[2em]
(b)_l & \displaystyle \kappa_1\psi_l=\frac{5}{6}(n+5)\psi_l^2 + \frac56\sum_{i=1, i \not= l}^n\psi_i^2,\\[2em]
(c) & \displaystyle \kappa_2=\sum\psi_i^2,\\[2em]
(d) & \displaystyle \kappa_{1,1}=\frac{5}{3}(n+5)\sum\psi_i^2.\\
\end{array}
$$
We have the following operations giving them from Pixton's relations
\begin{align*}
(a)_{k,l}&=(1)+\frac{3}{2} \cdot(2)_k+\frac32 \cdot (2)_l-4 \cdot (5),\\
(b)_l&=-\frac{2n+1}{3}\cdot (1)-(n-2) \cdot (2)_l-\sum_{i=1,i\neq l}^n(2)_i+\frac{8}{3}\sum_{i=1,i\neq l}^n(5)_{l,i},\\
(c)&=
\frac74 \cdot (1)-\frac{3}{16} \cdot (3)+\frac{3}{16}\sum_{i=1}^n(2)_i,\\
(d)&=
-\frac{2n^2+4n +33}3 \cdot (1)-\frac{8n-9}4 \sum_{i=1}^n(2)_i+\frac{7}{4} \cdot (3)\\
&+\frac{16}{3}\sum_{i<j}(5)_{i,j}.
\end{align*}

Thus we obtained BSZ's relations as linear combinations of Pixton's relations with polynomial coefficients. In particular, this proves that Pixton's relations form a complete family in genus 3.
\end{proof}

\section{The tautological ring in genus 4} 

\begin{prop}
We have
\begin{align*}
R^0(\cM_{4,n}) &= \Q,\\
R^1(\cM_{4,n}) &= \Q \langle \kappa_1, \psi_1, \dots, \psi_n \rangle,\\
R^3(\cM_{4,n}) &= \Q \langle \psi_1^3, \dots, \psi_n^3 \rangle,\\
R^d(\cM_{4,n}) &= 0 \quad \mbox{ for } d \geq 3.
\end{align*}
The Buryak-Shadrin-Zvonkine relations listed below form a complete family of relations is degree~3:
$$
\begin{array}{ll}
\kappa_3 & \displaystyle =  \;\; \sum_{i=1}^n \psi_i^3, \\[1.4em]
\kappa_{2,1} & \displaystyle = \;\; \frac73(n+7) \sum_{i=1}^n \psi_i^3,\\[1.4em]
\kappa_{1,1,1} & \displaystyle =\;\; \frac{35}9(n+7)(n+8) \sum_{i=1}^n \psi_i^3, \\[1.4em]
\kappa_2 \psi_k & \displaystyle =\;\; \frac79(n+8) \psi_k^3 + \frac{14}9 \sum_{i\not= k}^n \psi_i^3,\\[1.4em]
\kappa_{1,1} \psi_k & \displaystyle =\;\; \frac{35}{27}(n+7)(n+8) \psi_k^3+
\frac{70}{27} (n+7) \sum_{i\not= k} \psi_i^3,\\[1.4em]
\kappa_1 \psi^2_k & \displaystyle =\;\; \frac79(2n+13) \psi_k^3 + \frac79 \sum_{i\not= k}^n \psi_i^3,\\[1.4em]
\kappa_1 \psi_k \psi_l & \displaystyle =\;\; \frac{35}{27}(n+7) (\psi_k^3 + \psi_l^3) + \frac{35}{27} \sum_{i\not= k,l}^n \psi_i^3,\\[1.4em]
\psi_k^2 \psi_l & \displaystyle =\;\; \frac{14}9 \psi_k^3 + \frac79 \psi_l^3,\\[1.4em]
\psi_k \psi_l \psi_p& \displaystyle =\;\; \frac{35}{27}( \psi_k^3 + \psi_l^3+\psi_p^3).\\
\end{array}
$$
\end{prop}

\begin{proof}
The vanishing for $d \geq g$ is known by~\cite{Ion1}. The rank and the relations in $R^{g-1}(\cM_{g,n})$ are known by~\cite{BSZ}. By Mumford's stability~\cite{MadWei} there are no relations in $R^1(\cM_{4,n})$.
\end{proof}

It remains to study the group $R^2(\cM_{4,n})$. This group is spanned by the tautological classes $\kappa_2,\kappa_{1,1},\psi_i\kappa_1,\psi_i^2, \psi_i\psi_j$, where $i$ and $j$ are integers from $1$ to $n$. 

\begin{thm}
The classes $\psi_i\psi_j$, $1 \leq i < j \leq n$, $\psi_i^2$, $1 \leq i \leq n$, and $\kappa_{1,1}$ form a basis of $R^2(\cM_{4,n})$.
\end{thm}

The proof of this theorem is the goal of this section.

\subsection{Pixton's relations in genus~4}

Pixton's relations between these classes are determined by integers $a_1,\dots,a_n$ and a partition $\sigma$ such that
$$
3d \geq g+1+\sum_{i=1}^na_i+|\sigma|
$$ 
and the left-hand side has the same parity as the right-hand side. In our case these conditions boil down to
$$
\sum_{i=1}^na_i+|\sigma| = 1.
$$
Thus we have the following possibilities
\begin{itemize}
\item $\sigma=\{ 1\}$ and $a_i=0$ for all $i$.
\item $\sigma=\emptyset$, $a_k=1$ for some $1\leq k\leq n$ and $a_i=0$ for $i\neq k$.
\end{itemize}

In the first case, the relation is 
\begin{align*}
\Big[\kappa(\exp(\{ -60T-27720T^2\})&\cdot\{ -1+84T+32760T^2\})\\
&\cdot\prod_{i=1}^n(1+60\psi_iT+27720\psi_i^2T^2)\Big]_{T^2}
\end{align*}
which gives
\begin{align*}
&\Big[\kappa(exp(1-60K_1T-27720K_2T^2+1800K_1^2T^2) \\
&\cdot(-K_0+84K_1T+32760K_2T^2))\cdot\prod_{i=1}^n(1+60\psi_iT+27720\psi_i^2T^2)\Big]_{T^2}.&
\end{align*}
From this we obtain
\begin{align*}
&\Big[(-\kappa_0+(84+60(1+\kappa_0))\kappa_1T+(32760+27720(1+\kappa_0))\kappa_2T^2\\
&-(5040+1800(2+\kappa_0))\kappa_{1,1}T^2)\cdot\prod_{i=1}^n(1+60\psi_iT+27720\psi_i^2T^2)\Big]_{T^2}.&
\end{align*}
Extracting the coefficient of  degree $2$ and dividing it by $360$ we get
\begin{align*}
&(630-77\kappa_0)\kappa_2-(24+5\kappa_0)\kappa_{1,1}-77\kappa_0\sum_{i=1}^n\psi_i^2\\
&+(24+10\kappa_0)\sum_{i=1}^n\kappa_1\psi_i-10\kappa_0\sum_{i<j}\psi_i\psi_j.
\end{align*}
We see that the coefficients depend on $n$ since $\kappa_0=2g-2+n$\\
$=6+n$.\\

In the second case, if $\sigma$ is empty, one of the $a_i$'s is 1 and the others are zero we get the following relation
\begin{align*}
\Big[\kappa(\exp(\{ -60T-27720T^2\}))\cdot(-1+84\psi_kT+32760\psi_k^2T^2)&\\
\cdot\prod_{i=1,i\neq k}^n(1+60\psi_iT+27720\psi_i^2T^2)\Big]_{T^2}&
\end{align*}
which gives 
\begin{align*}&\Big[(1-60\kappa_1T-27720\kappa_2T^2+1800\kappa_{1,1}T^2)\\
&\cdot(-1+84\psi_kT+32760\psi_k^2T^2)\cdot\prod_{i=1,i\neq k}^n(1+60\psi_iT+27720\psi_i^2T^2)\Big]_{T^2}.&
\end{align*}
Extracting the coefficient of $T^2$ and dividing by $360$, we get
\begin{align*}
77\kappa_2-5\kappa_{1,1}+91\psi_k^2-77\sum_{i=1,i\neq k}^n\psi_i^2-14\kappa_1\psi_k+10\sum_{i=1,i\neq k}^n\kappa_1\psi_i\\
+14\sum_{i=1,i\neq k}^n\psi_k\psi_i-10\sum_{i<j,i,j\neq k}\psi_i\psi_j.
\end{align*}

\subsection{Upper bound for the dimension}
 Let's write the matrix whose lines are composed by the coefficients of the classes $\kappa_2,\kappa_1\psi_1,...,\kappa_1\psi_n$ in that order in the relations in the same order as previous. We see that, in the first line of this matrix, the coefficient of $\kappa_2$ can be written $1092+77n$ and that the coefficients $\kappa_1\psi_i$ can be written $84+10n$. For a fixed integer,$n$ the matrix will be
$$
\begin{pmatrix}
 1092+77n & 84+10n & \cdots & \cdots & 84+10n \\
77        & -14    & 10     & \cdots & 10     \\
\vdots    & 10     & \ddots & \ddots & \vdots \\
\vdots    & \vdots & \ddots & \ddots & 10     \\
77        & 10     & \dots  &  10    &-14
\end{pmatrix}
$$

We can divide the first column by $7$ to obtain the following matrix
$$
\begin{pmatrix}
 156+11n & 84+10n & \cdots & \cdots & 84+10n \\
11        & -14    & 10     & \cdots & 10     \\
\vdots    & 10     & \ddots & \ddots & \vdots \\
\vdots    & \vdots & \ddots & \ddots & 10     \\
11        & 10     & \dots  &  10    &-14
\end{pmatrix}
$$

Adding the first column multiplied by $\displaystyle{-\frac{10}{11}}$ to all the other columns we obtain 
$$
\begin{pmatrix}
 156+11n & -\frac{636}{11} & \cdots & \cdots & -\frac{636}{11}    \\
11        & -24              & 0      & \cdots & 0                  \\
\vdots    & 0               & \ddots & \ddots & \vdots             \\
\vdots    & \vdots          & \ddots & \ddots & 0                  \\
11        & 0               & \dots  & 0      & -24
\end{pmatrix}
$$

Now we add all lines but the fisrt multiplied by $\displaystyle{-\frac{53}{22}}$ to the first line we get the following matrix
$$
\begin{pmatrix}
 156-\frac{31}{22}n& 0      & \cdots & \cdots & 0      \\
11                  & -24     & \ddots &        & \vdots \\
\vdots              & 0      & \ddots & \ddots & \vdots \\
\vdots              & \vdots & \ddots & \ddots & 0      \\
11                  & 0      & \dots  & 0      & -24
\end{pmatrix}
$$

The determinant of this matrix, $\displaystyle{(-24)^n\cdot(156-\frac{31}{22}n)}$, which only vanishes for $\displaystyle{n=\frac{31\cdot 156}{22}}\notin\N$, hence can not vanish for any integer $n$, this matrix is then invertible for all $n$ and in the system of $n+1$ relations thus the classes $\kappa_2,\kappa_1\psi_1,...,\kappa_1\psi_n$ can be expressed in terms of the classes $\kappa_{1,1}$ and $\psi_i\psi_j$, for $1\leq i,j\leq n$.

\subsection{The matrix M}
We want to show that the classes $\kappa_{1,1},\psi_i^2,\psi_i\psi_j$ don't have any relation between them. We consider now the classes $\lambda_4\lambda_3\prod_{i=1}^n\psi_i$, $\lambda_4\lambda_3\kappa_1\prod_{i=1,i\neq k}^n\psi_i$, for integers $k$ such that $1\leq k\leq n$ , and $\lambda_4\lambda_3\psi_k^2\prod_{i=1,i\neq k,l}^n\psi_i$, where $1\leq k<l\leq n$.\\

\begin{prop}
These classes are defined on $\oM_{4,n}$ and vanish on its boundary.
\end{prop}

\begin{proof}
The part $\lambda_4\lambda_3$ makes these classes vanish on the whole boundary except on the curves having a separating node whith a branch of genus $0$. The vanishing on these strata comes from the vanishing of $\prod_{i=1}^n\psi_i$, $\kappa_1\prod_{i=1,i\neq k}^n\psi_i$ and $\psi_k\prod_{i=1,i\neq l}^n\psi_i$ on $\oM_{0,n+1}$.
\end{proof}

In the following, we will use the following notations
\begin{itemize}
\item $\Psi_{\emptyset}:=\prod_{i=1}^n\psi_i$.
\item For $k$ a positive integer such that $1\leq k\leq n$, we write $\Psi_k:=\prod_{i=1,i\neq k}^n\psi_i$.
\item For $k,l$ distinct positive integers such that $1\leq k<l\leq n$, we write $\Psi_{\{ k,l\}}:=\psi_k^2\prod_{i=1,i\neq k,l}\psi_i$.
\end{itemize}

We build now the matrix $M$ composed by the intersection numbers given by the products classes of the two lists.

$$\lambda_4\lambda_3\Psi_{\emptyset}~~~~\lambda_4\lambda_3\kappa_1\Psi_1~\cdots~\lambda_4\lambda_3\kappa_1\Psi_n~~~~\lambda_4\lambda_3\Psi_{\{ 1,2\}}~\cdots~\lambda_4\lambda_3\Psi_{\{ n-1,n\}}~~~~~~$$

$$
\begin{pmatrix}
 &~~~~~~~~~~~~~~~ &~~~~~~~~~~~~&~~~~ &~~~~~~&~~~~ &~~~~~~~~~~ &~~~~~~~~& \\
 &~~~~~~~~~~~~~~~ &~~~~~~~~~~~~&~~~~ &~~~~~~&M_{\alpha\beta} &~~~~~~~~~&~~~~~~~~~~ & \\
 &~~~~~~~~~~~~~~~ &~~~~~~~~~~~~&~~~~ &~~~~~~&~~~~ &~~~~~~~~~~ &~~~~~~~~& \\
 &~~~~~~~~~~~~~~~ &~~~~~~~~~~~~&~~~~ &~~~~~~&~~~~ &~~~~~~~~~~ &~~~~~~~~& \\
 &~~~~~~~~~~~~~~~ &~~~~~~~~~~~~&~~~~ &~~~~~~&~~~~ &~~~~~~~~~~ &~~~~~~~~& \\
 &~~~~~~~~~~~~~~~ &~~~~~~~~~~~~&~~~~ &~~~~~~&~~~~ &~~~~~~~~~~ &~~~~~~~~& \\
 &~~~~~~~~~~~~~~~ &~~~~~~~~~~~~&~~~~ &~~~~~~&~~~~ &~~~~~~~~~~ &~~~~~~~~& \\
 &~~~~~~~~~~~~~~~ &~~~~~~~~~~~~&~~~~ &~~~~~~&~~~~ &~~~~~~~~~~ &~~~~~~~~& 
\end{pmatrix}
\begin{matrix*}[l]
\kappa_{1,1}\\
\psi_1^2\\
\vdots\\
\psi_n^2\\
\psi_1\psi_2\\
\vdots\\
\psi_{n-1}\psi_n
\end{matrix*}$$

where $M_{\alpha\beta}=\displaystyle{\int_{\oM_{4,n}}\alpha\cdot\beta}$. We are able to calculate these integrals using the following result.

\begin{prop}
Let $\pi:\cM_{g,n}^{rt}\rightarrow\cM_g$ be the forgetful map forgetting all marked points and let $d_1,...,d_n,k_1,...,k_m$ be nonnegative integers. Then in $R^{g-2}(\cM_g)$, we have
$$\pi_*(\kappa_{k_1,...,k_m}\prod_{i=1}^n\psi_i^{d_i+1})=\frac{(2g-3+n+m)!(2g-3)!!}{(2g-2)!\prod_{i=1}^n(2d_i+1)!!\prod_{i=1}^m(2k_i+1)!!}\kappa_{g-2}.$$
\end{prop}

That result is proved in \cite{BSZ}.

\subsubsection{Coefficients}

Let's number the lines of $M$ in the following way: $\emptyset$ designates the first line,  the $n$ following lines will be denoted by an integer $i$ from $1$ to $n$ and the other lines by $\lbrace i,j\rbrace$, where $i$ and $j$ are integers such that $1\leq i<j\leq n$. We denote the columns of $M$ in the same way. It might be useful to define the set of these coordinates, $\cE:=\{\emptyset,1,...,n,\{1,2\},...,\{n-1,n\}\}$. Logically, for a column vector $u$ of dimension $\frac{n(n+1)}{2}+1$, for $k\in\cE$ designating a line we denote $u_k$ the entry of the vector $u$ at the line $k$. In the case $k=\{i,j\}$, we may allow ourselves to write $u_{i,j}$, we simply never do this abuse for coefficients of matrices since something like $M_{1,2,3}$ would yield to a lot of confusion.\

In the following, the letters $k,l,p,q$ will always denote positive integers such that
\begin{center}
$1\leq k<l\leq n$\\
$1\leq p<q\leq n$
\end{center}

and $\alpha,\beta\in\cE$, $\alpha$ standing for the lines and $\beta$ for the columns.

The last proposition enables us to calculate the coefficients of $M$.
\paragraph{\textbf{First line}}
\begin{itemize}
\item \textbf{$\alpha=\beta=\emptyset$}
\begin{align*}
M_{\emptyset\emptyset}&=\int_{\oM_{4,n}}\lambda_4\lambda_3\cdot\prod_{i=1}^n\psi_i\cdot\kappa_{1,1}\\
                      &=\frac{(n+7)!\cdot 5!!}{6!\cdot\prod_{i=1}^n1!!\cdot3!!\cdot3!!}\cdot\int_{\oM_4}\lambda_4\lambda_3\kappa_2\\
                      &=\frac{(n+7)!\cdot5}{6!\cdot3}\cdot\int_{\oM_4}\lambda_4\lambda_3\kappa_2\\
                      &=\frac{(n+7)!}{2^4\cdot3^3}\cdot\int_{\oM_4}\lambda_4\lambda_3\kappa_2.\\
\end{align*}

\item \textbf{$\alpha=\emptyset$ and $\beta=p$}
\begin{align*}
M_{\emptyset,p}&=\int_{\oM_{4,n}}\lambda_4\lambda_3\kappa_1\cdot\prod_{i=1,i\neq p}^n\psi_i\cdot\kappa_{1,1}\\
                &=\int_{\oM_{4,n}}\lambda_4\lambda_3(\kappa_{1,1,1}-2\kappa_{1,2})\cdot\prod_{i=1,i\neq p}^n\psi_i
\end{align*}
because $\kappa_1\kappa_{1,1}=\kappa_{1,1,1}-2\kappa_{1,2}$, this intersection number is then the sum of two terms which we calculate separately  
\begin{align*}
\int_{\oM_{4,n}}\lambda_4\lambda_3\kappa_{1,1,1}\cdot\prod_{i=1,i\neq p}^n\psi_i&=\frac{(n+8)!\cdot 5!!}{6!\cdot 3!!\cdot 3!!\cdot3!!}\cdot\int_{\oM_4}\lambda_4\lambda_3\kappa_2\\
                                                                           &=\frac{(n+8)!}{2^4\cdot 3^4}\cdot\int_{\oM_4}\lambda_4\lambda_3\kappa_2.
\end{align*}
The second term is 

\begin{align*}
-2\cdot\int_{\oM_{4,n}}\lambda_4\lambda_3\kappa_{1,2}\cdot\prod_{i=1,i\neq p}^n\psi_i&=-2\cdot\frac{(n+7)!\cdot5!!}{6!\cdot 3!!\cdot 5!!}\cdot\int_{\oM_4}\lambda_4\lambda_3\kappa_2\\
                                                                           &=-2\cdot\frac{(n+7)!}{2^4\cdot 3^3\cdot 5}\cdot\int_{\oM_4}\lambda_4\lambda_3\kappa_2.
\end{align*}
Hence we have
\begin{align*}
M_{\emptyset,p}&=\frac{(n+8)!}{2^4\cdot3^4}\cdot\int_{\oM_4}\lambda_4\lambda_3\kappa_2-2\cdot\frac{(n+7)!}{3\cdot6!}\cdot\int_{\oM_4}\lambda_4\lambda_3\kappa_2 \\        
         &=\frac{(5n+34)\cdot(n+7)!}{2^4\cdot 3^4\cdot5}\cdot\int_{\oM_4}\lambda_4\lambda_3\kappa_2.
\end{align*}

\item \textbf{$\alpha=\emptyset$ and $\beta=\lbrace p,q\rbrace$}
\begin{align*}
M_{\emptyset,\{ p,q\}}&=\int_{\oM_{4,n}}\lambda_4\lambda_3\psi_p^2\cdot\prod_{i=1,i\neq p,q}^n\psi_i\cdot\kappa_{1,1}\\
       &=\frac{(n+7)!\cdot 5!!}{6!\cdot 3!!\cdot\prod_{i=1,i\neq k,l}^n1!!\cdot 3!!\cdot 3!!}\cdot\int_{\oM_4}\lambda_4\lambda_3\kappa_2\\
       &=\frac{(n+7)!\cdot 3\cdot 5}{6!\cdot 3\cdot 3\cdot 3}\cdot\int_{\oM_4}\lambda_4\lambda_3\kappa_2\\
       &=\frac{(n+7)!}{2^4\cdot 3^4}\cdot\int_{\oM_4}\lambda_4\lambda_3\kappa_2.\\
\end{align*}
\end{itemize}

\paragraph{\textbf{Lines of type $\alpha=k$}}
\begin{itemize}
\item \textbf{$\alpha=k$ and $\beta=\emptyset$}
\begin{align*}
M_{k,\emptyset}&=\int_{\oM_{4,n}}\lambda_4\lambda_3\cdot\prod_{i=1}^n\psi_i\cdot\psi_k^2\\
&=\int_{\oM_{4,n}}\lambda_4\lambda_3\cdot\prod_{i=1,i\neq k}^n\psi_i\cdot\psi_k^3\\
&=\frac{(n+5)!\cdot 5!!}{6!\cdot\prod_{i=1,i\neq k}^n1!!\cdot 5!!}\cdot\int_{\oM_4}\lambda_4\lambda_3\kappa_2\\
&=\frac{(n+5)!\cdot 5!!}{6!\cdot 5!!}\cdot\int_{\oM_4}\lambda_4\lambda_3\kappa_2\\
&=\frac{(n+5)!}{2^4\cdot 3^2\cdot 5}\cdot\int_{\oM_4}\lambda_4\lambda_3\kappa_2.\\
\end{align*}

\item \textbf{$\alpha=k$ and $\beta=p$}
\begin{itemize}
\item When $k\neq p$
\begin{align*}
M_{k,p}&=\int_{\oM_{4,n}}\lambda_4\lambda_3\kappa_1\cdot\prod_{i=1,i\neq p}^n\psi_i\cdot\psi_k^2\\
&=\int_{\oM_{4,n}}\lambda_4\lambda_3\kappa_1\cdot\prod_{i=1,i\neq p,k}^n\psi_i\cdot\psi_k^3\\
&=\frac{(n+6)!\cdot 5!!}{6!\cdot\prod_{i=1,i\neq p,k}^n1!!\cdot 3!!\cdot 5!!}\cdot\int_{\oM_4}\lambda_4\lambda_3\kappa_2\\
&=\frac{(n+6)!}{2^4\cdot 3^3\cdot 5}\cdot\int_{\oM_4}\lambda_4\lambda_3\kappa_2.\\
\end{align*}

\item When $k=p$
\begin{align*}
M_{k,k}&=\int_{\oM_{4,n}}\lambda_4\lambda_3\kappa_1\cdot\prod_{i=1,i\neq k}^n\psi_i\cdot\psi_k^2\\
&=\frac{(n+6)!\cdot 5!!}{6!\prod_{i=1,i\neq k}^n1!!\cdot 3!!\cdot 3!!}\cdot\int_{\oM_4}\lambda_4\lambda_3\kappa_2\\
&=\frac{(n+6)!\cdot 5}{6!\cdot 3}\cdot\int_{\oM_4}\lambda_4\lambda_3\kappa_2\\
&=\frac{(n+6)!}{2^4\cdot 3^3}\cdot\int_{\oM_4}\lambda_4\lambda_3\kappa_2.\\
\end{align*}
\end{itemize}

\item \textbf{$\alpha=k$ and $\beta=\lbrace p,q\rbrace$}
\begin{itemize}
\item When $k\neq p,q$

\begin{align*}
M_{k,\{ p,q\}}&=\int_{\oM_{4,n}}\lambda_4\lambda_3\psi_p^2\cdot\prod_{i=1, i\neq p,q}^n\psi_i\cdot\psi_k^2\\
       &=\int_{\oM_{4,n}}\lambda_4\lambda_3\psi_p^2\psi_k^3\cdot\prod_{i=1,i\neq k,p,q}^n\psi_i\\
       &=\frac{(n+5)!\cdot 5!!}{6!\cdot 5!!\cdot 3!!}\cdot\int_{\oM_4}\lambda_4\lambda_3\kappa_2\\
       &=\frac{(n+5)!}{2^4\cdot 3^3\cdot 5}\cdot \int_{\oM_4}\lambda_4\lambda_3\kappa_2.
\end{align*}

\item When $k=p$
\begin{align*}
M_{k,\{ k,q\}}&=\int_{\oM_{4,n}}\lambda_4\lambda_3\psi_p^2\cdot \prod_{i=1, i\neq p,q}^n\psi_i\cdot\psi_p^2\\
       &=\int_{\oM_{4,n}}\lambda_{4,n}\lambda_3\psi_p^4\cdot \prod_{i=1, i\neq p,q}^n\psi_i\\
       &=\frac{(n+5)!\cdot 5!!}{6!\cdot 7!!}\cdot\int_{\oM_4}\lambda_4\lambda_3\kappa_2\\
       &=\frac{(n+5)!}{2^4\cdot 3^2\cdot 5\cdot 7}\cdot\int_{\oM_4}\lambda_4\lambda_3\kappa_2.
\end{align*}

\item When $k=q$
\begin{align*}
M_{k,\{ p,k\}}&=\int_{\oM_{4,n}}\lambda_4\lambda_3\psi_p^2\cdot\prod_{i=1, i\neq p,q}^n\psi_i\cdot\psi_q^2\\
       &=\frac{(n+5)!\cdot 5!!}{6!\cdot 3!!\cdot 3!!}\cdot\int_{\oM_4}\lambda_4\lambda_3\kappa_2\\
       &=\frac{(n+5)!}{2^4\cdot 3^3}\cdot\int_{\oM_4}\lambda_4\lambda_3\kappa_2.
\end{align*}
\end{itemize}
\end{itemize}

\paragraph{\textbf{Lines of type $\alpha=\lbrace k,l\rbrace$}}
\begin{itemize}
\item \textbf{$\alpha=\lbrace k,l\rbrace$ and $\beta=\emptyset$}
\begin{align*}
M_{\{ k,l\},\emptyset}&=\int_{\oM_{4,n}}\lambda_4\lambda_3\cdot\prod_{i=1}^n\psi_i\cdot\psi_k\psi_l\\
       &=\int_{\oM_{4,n}}\lambda_4\lambda_3\cdot\prod_{i=1,i\neq k,l}^n\psi_i\cdot\psi_k^2\psi_l^2\\
       &=\frac{(n+5)!\cdot 5!!}{6!\cdot 3!!\cdot 3!!}\cdot\int_{\oM_4}\lambda_4\lambda_3\kappa_2\\
       &=\frac{(n+5)!}{2^4\cdot 3^3}\cdot\int_{\oM_4}\lambda_4\lambda_3\kappa_2.
\end{align*}

\item \textbf{$\alpha=\{ k,l\}$ and $\beta=p$}
\begin{itemize}
\item When $p\neq k,l$
\begin{align*}
M_{\{ k,l\},p}&=\int_{\oM_{4,n}}\lambda_4\lambda_3\kappa_1\cdot\prod_{i=1,i\neq p}^n\psi_i\cdot\psi_k\psi_l\\
       &=\int_{\oM_{4,n}}\lambda_4\lambda_3\kappa_1\cdot\prod_{i=1,i\neq p,k,l}^n\psi_i\cdot\psi_k^2\psi_l^2\\
       &=\frac{(n+6)!\cdot 5!!}{6!\cdot 3!!\cdot 3!!\cdot 3!!}\cdot\int_{\oM_4}\lambda_4\lambda_3\kappa_2\\
       &=\frac{(n+6)!}{2^4\cdot 3^4}\cdot\int_{\oM_4}\lambda_4\lambda_3\kappa_2.
\end{align*}

\item When $p=k$ or $p=l$
\begin{align*}
M_{\{ k,l\},k}&=M_{\{ k,l\},l}=\int_{\oM_{4,n}}\lambda_4\lambda_3\kappa_1\cdot\prod_{i=1,i\neq k}^n\psi_i\cdot\psi_k\psi_l\\
       &=\int_{\oM_{4,n}}\lambda_4\lambda_3\kappa_1\cdot\prod_{i=1,i\neq l}^n\psi_i\psi_l^2\\
       &=\frac{(n+6)!\cdot 5!!}{6!\cdot 3!!\cdot 3!!}\cdot\int_{\oM_4}\lambda_4\lambda_3\kappa_2\\
       &=\frac{(n+6)!}{2^4\cdot 3^3}\cdot \int_{\oM_4}\lambda_4\lambda_3\kappa_2.
\end{align*}

\end{itemize}

\item \textbf{$\alpha=\{ k,l\}$ and $\beta=\{ p,q\}$}
\begin{itemize}
\item When $\{ k,l\}\cap\{ p,q\}$ is empty
\begin{align*}
M_{\{ k,l\},\{ p,q\}}&=\int_{\oM_{4,n}}\lambda_4\lambda_3\psi_p^2\cdot\prod_{i=1,i\neq p,q}^n\psi_i\cdot\psi_p\psi_q\\
       &= \int_{\oM_{4,n}}\lambda_4\lambda_3\psi_p^2\cdot\prod_{i=1,i\neq k,l,p,q}^n\psi_i\cdot\psi_k^2\psi_l^2\\
       &=\frac{(n+5)!\cdot 5!!}{6!\cdot 3!!\cdot 3!!\cdot 3!!}\cdot\int_{\oM_4}\lambda_4\lambda_3\kappa_2\\
       &=\frac{(n+5)!}{2^4\cdot 3^4}\cdot\int_{\oM_4}\lambda_4\lambda_3\kappa_2.\\
\end{align*}
\item When $\{ k,l\}\cap\{ p,q\}=\{ p\}$
\begin{align*}
M_{\{ p,l\},\{ p,q\}}&=\int_{\oM_{4,n}}\lambda_4\lambda_3\psi_p^2\cdot\prod_{i=1,i\neq p,q}^n\psi_i\cdot\psi_p\psi_l\\
      &=\int_{\oM_{4,n}}\lambda_4\lambda_3\psi_p^3\cdot\prod_{i=1,i\neq k,l,p,q}^n\psi_i\cdot\psi_l^2\\
      &=\frac{(n+5)!\cdot 5!!}{6!\cdot 3!!\cdot 5!!}\cdot\int_{\oM_4}\lambda_4\lambda_3\kappa_2\\
      &=\frac{(n+5)!}{2^4\cdot 3^3\cdot 5}\cdot\int_{\oM_4}\lambda_4\lambda_3\kappa_2
\end{align*}
and 
$$M_{\{ k,l\},\{ p,k\}}=\frac{(n+5)!}{2^4\cdot 3^3\cdot 5}\cdot\int_{\oM_4}\lambda_4\lambda_3\kappa_2.$$

\item When $\{ k,l\}\cap\{ p,q\}=\{ q\}$
\begin{align*}
M_{\{ q,l\},\{ p,q\}}&=\int_{\oM_{4,n}}\lambda_4\lambda_3\psi_p^2\cdot\prod_{i=1,i\neq p,q}^n\psi_i\cdot\psi_q\psi_l\\
      &=\int_{\oM_{4,n}}\lambda_4\lambda_3\psi_p^2\psi_l^2\cdot\prod_{i=1,i\neq q,p,l}^n\psi_i\\
      &=\frac{(n+5)!\cdot 5!!}{6!\cdot 3!!\cdot 3!!}\cdot\int_{\oM_4}\lambda_4\lambda_3\kappa_2\\
      &=\frac{(n+5)!}{2^4\cdot 3^3}\cdot\int_{\oM_4}\lambda_4\lambda_3\kappa_2
\end{align*}
and 
$$M_{\{ k,q\},\{ p,q\}}=\frac{(n+5)!}{2^4\cdot 3^3}\cdot\int_{\oM_4}\lambda_4\lambda_3\kappa_2.$$

\item When $\lbrace k,l\rbrace\cap\lbrace p,q\rbrace=\lbrace p,q\rbrace$
\begin{align*}
M_{\{ p,q\},\{ p,q\}}&=\int_{\oM_{4,n}}\lambda_4\lambda_3\psi_p^2\cdot\prod_{i=1,i\neq p,q}^n\psi_i\cdot\psi_p\psi_q\\
      &=\int_{\oM_{4,n}}\lambda_4\lambda_3\psi_p^3\cdot\prod_{i=1,i\neq p}^n\psi_i\\
      &=\frac{(n+5)!\cdot 5!!}{6!\cdot 5!!}\cdot\int_{\oM_4}\lambda_4\lambda_3\kappa_2\\
      &=\frac{(n+5)!}{2^4\cdot 3^2\cdot 5}\cdot\int_{\oM_4}\lambda_4\lambda_3\kappa_2.
\end{align*}
\end{itemize}
\end{itemize}

\subsubsection{Simplification of M, \^{M}}

The integral $\displaystyle{\int_{\oM_4}\lambda_4\lambda_3\kappa_2}$ is non zero since, for any genus $g$, we have from \cite{BSZ}
$$\int_{\oM_g}\lambda_g\lambda_{g-1}\kappa_{g-2}=\frac{(-1)^{g-1}B_{2g}(g-1)!}{2^g(2g)!}.$$
Recalling that $B_{2g}$ are the Bernouilli numbers given by 
$$\frac{1}{1-e^{-x}}=1+\frac{x}{2}+\sum_{n\geq 1}\frac{B_{2n}}{(2n)!}x^{2n}.$$
We have $B_8=\frac{-1}{30}$, so we can even calculate
$$\int_{\oM_4}\lambda_4\lambda_3\kappa_2=\frac{-1}{2^{11}\cdot 3^2\cdot 5^2\cdot 7},$$
in $R^9(\oM_4)\simeq \Q$.
In order to simplify the calculations, we define a matrix $\hM$ by dividing $M$ by $\displaystyle{\frac{(n+5)!}{2^4\cdot 3}\cdot\int_{\oM_4}\lambda_4\lambda_3\kappa_2}$ and by the following operations on the columns and raws of $M$
\begin{itemize}
\item The line $\alpha=\emptyset$ is divided by $n+7$.
\item The lines $\alpha=i$, for $1\leq i\leq n$, are divided by $\frac{3}{7}$.
\item The column $\beta=\emptyset$ is divided by $3$.
\item The column $\beta=i$, for $1\leq i\leq n$, are divided by $n+6$.
\end{itemize}

Here are the coefficients of $\hM$.

\paragraph{\textbf{First line}}

\begin{itemize}
\item \textbf{$\alpha=\beta=\emptyset$}

$$\hM_{\emptyset,\emptyset}=5(n+6).$$

\item \textbf{$\alpha=\emptyset$ and $\beta=p$}

$$\hM_{\emptyset,k}=5n+34.$$

\item \textbf{$\alpha=\emptyset$ and $\beta=\{ k,l\}$}

$$\widehat{M}_{\emptyset,k\{ k,l\}}=5(n+6).$$
\end{itemize}

\paragraph{\textbf{Lines of type $\alpha=k$}}

\begin{itemize}
\item \textbf{$\alpha=k$ and $\beta=\emptyset$}
$$\hM_{k,\emptyset}=7.$$

\item \textbf{$\alpha=k$ and $\beta=p$}
\begin{itemize}
\item When $k\neq p$
$$\hM_{k,p}=7.$$

\item When $k=p$

$$\hM_{k,k}=35.$$
\end{itemize}

\item \textbf{$\alpha=k$ and $\beta=\{ p,q\}$}
\begin{itemize}
\item When $k\neq p,q$
$$\hM_{k,\{ p,q\}}=7.$$

\item When $k=p$
$$\hM_{k,\{ k,q\}}=3.$$

\item When $k=q$
$$\hM_{k,\{ p,k\}}=35.$$

\end{itemize}
\end{itemize}

\paragraph{\textbf{Lines of type $\alpha=\{ k,l\}$}}

\begin{itemize}
\item \textbf{$\alpha=\{ k,l\}$ and $\beta=\emptyset$}

$$\hM_{\{ k,l\},\emptyset}=5.$$

\item \textbf{$\alpha=\{ k,l\}$ and $\beta=p$}
\begin{itemize}
\item When $p\neq k,l$
$$\hM_{\{ k,l\},p}=5.$$
\item When $p=k$ (or $p=l$)
$$\hM_{\{ k,l\},k}=15.$$
\end{itemize}

\item \textbf{$\alpha=\{ k,l\}$ and $\beta=\{ p,q\}$}
\begin{itemize}
\item When $\{ k,l\}\cap\{ p,q\}$ is empty
$$\hM_{\{ k,l\},\{ p,q\}}=5.$$

\item When $\{ k,l\}\cap\{ p,q\}=\{ p\}$
$$\hM_{\{ p,l\},\{ p,q\}}=3,$$
$$\hM_{\{ k,p\},\{ p,q\}}=3.$$

\item When $\{ k,l\}\cap\{ p,q\}=\{ q\}$
$$\hM_{\{ q,l\},\{ p,q\}}=15,$$
$$\hM_{\{ k,q\},\{ p,q\}}=15.$$

\item When $\{ k,l\}\cap\{ p,q\}=\{p,q\}$
$$\hM_{\{ p,q\},\{ p,q\}}=9.$$
\end{itemize}
\end{itemize}

\subsection{Rank of M}

We look at $M$ as the matrix of an endomorphism of a $\Q$-vector space $E$ of dimension $N = 1 + n + n(n-1)/2$. A vector of $E$ has coordinates $\alpha$, $\beta_i$ for $1 \leq i \leq n$, and $\gamma_{ij}$ with $1 \leq i < j \leq n$. 

Calculations via Maple shows that, for $n\leq 18$, the characteristic polynomial of $\hM$ is the product of degree $2$ polynomials of multiplicity $1$, a degree $3$ polynomial of multiplicity $1$ also (sometimes this ones splits into a two polynomials of degrees $1$ and $2$) and a polynomial of degree $1$ with a multiplicity which is a function of $n$. The decomposition via Maple is made only if it leads to polynomials in $\mathbb{Z}[X]$. This justifies the strategy for the calculation of the rank.

The polynomials of degree $2$ are associated to stable planes which we describe at first. Next we will look for the eigenvectors associated to the polynomial of degree $1$ and we will finish by the study of the stable $3$-dimensional space of $\hM$.

\subsubsection{Basis of stable planes}

For $1 \leq i \leq n-1$ introduce two vectors $u_i$ and $v_i$ in~$E$. These vectors have the following nonzero coordinates:
\begin{align*}
u_i &:  (\beta_i=1, \beta_{i+1} = -1)\\
v_i & :  (\gamma_{ik} = 1, \gamma_{i+1,k} = -1) \mbox{ for },~~k\not= i,i+1.
\end{align*}
All the coordinates that are not listed are equal to~0.

\begin{center}
\includegraphics[scale=0.4]{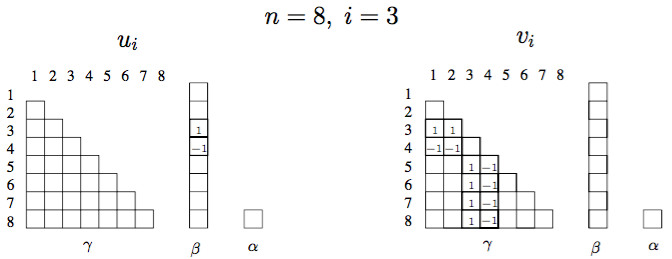}
\end{center}

\begin{prop}
The plane spanned by $u_i$ and $v_i$ is invariant under the action of~$M$ for every~$i$. Specifically, we have
\begin{align*}
M(u_i) & = 28u_i+10v_i\\
M(v_i) & = (32i-24-4n)u_i+(12i-12-2n)v_i\\
\end{align*}
\end{prop}

\begin{proof} We will calculate the products $M\cdot u_i$ and $M\cdot v_i$ to show that they are linear combinations of $u_i$ and $v_i$.
\paragraph{\textbf{Coordinate $\emptyset$}} 
We have
\begin{align*}
\big(\hM\cdot u_i\big)_{\emptyset}&=\hM_{\emptyset,i}-\hM_{\emptyset,i+1}\\
                              &=(5n+34)-(5n+34)\\
                              &=0.
\end{align*}

\paragraph{\textbf{Coordinates $k$}}

\begin{itemize}
\item If $k<i$, we have 
\begin{align*}
\big(\hM\cdot u_i\big)_k&=\hM_{k,i}-\hM_{k,i+1}\\
                  &=7-7\\
                  &=0.
\end{align*}

\item Similarly when $i+1<k$, we have
\begin{align*}
\big(\hM\cdot u_i\big)_k&=\hM_{i,k}-\hM_{i+1,k}\\
                        &=7-7\\
                        &=0.
\end{align*}

\item The cases $k=i$ and $k=i+1$, we have
\begin{align*}
\big(\hM\cdot u_i\big)_i&=\hM_{i,i}-\hM_{i+1,i}\\
                        &=35-7\\
                        &=28
\end{align*}
and
\begin{align*}
\big(\hM\cdot u_i)\big)_{i+1}&=\hM_{i,i+1}-\hM_{i+1,i+1}
                             &=7-35\\
                             &=-28.
\end{align*}                             
\end{itemize}

\paragraph{\textbf{Coordinates $\lbrace k,l\rbrace$}}
Now let's calculate the coordinates of type $\lbrace k,l\rbrace$ of $\big(M\cdot u_i\big)$. 
\begin{itemize}
\item For $k$ and $l$ distinct integers such that $1\leq k<l\leq n$, none of them being equal to $i$ or $i+1$, we have
\begin{align*}
\big(\hM\cdot u_i)\big)_{\lbrace k,l\rbrace}&=\hM_{\lbrace k,l\rbrace,i}-\hM_{\lbrace k,l\rbrace,i+1}\\
                                       &=15-15\\
                                       &=0
\end{align*}
\item  We also get 
$$\big(M\cdot u_i\big)_{\lbrace i,i+1\rbrace}=\hM_{\lbrace i,i+1\rbrace,i}-\hM_{\lbrace i,i+1\rbrace,i+1}=15-15=0.$$

\item For $k<i$, 

$$\big(M\cdot u_i\big)_{\lbrace k,i\rbrace}=\hM_{\lbrace k,i\rbrace,i}-\hM_{\lbrace k,i\rbrace,i+1}=15-5=10$$
$$\big(M\cdot u_i\big)_{\lbrace k,i+1\rbrace}=\hM_{\lbrace k,i\rbrace,i+1}-\hM_{\lbrace k,i\rbrace,i}=5-15=-10.$$
\item For $l>i+1$, 

$$\big(M\cdot u_i\big)_{\lbrace i,l\rbrace}=\hM_{\lbrace i,l\rbrace,a}-\hM_{\lbrace i,l\rbrace,i+1}=15-5=10$$
$$\big(M\cdot u_i\big)_{\lbrace i+1,l\rbrace}=\hM_{\lbrace i,l\rbrace,i+1}-\hM_{\lbrace i,l\rbrace,i}=5-15=-10.$$
\end{itemize}

In conclusion, we see that
$$\hM\cdot u_i=28u_i+10v_i.$$

~\\

We calculate the coordinates of $M\cdot v_i$ in the same way and we get 
\paragraph{\textbf{Coordinate $\emptyset$}}
We have
\begin{align*}
\big(\hM\cdot v_i\big)_{\emptyset}&=\sum_{p=1}^{i-1}\hM_{\emptyset,\lbrace p,i\rbrace}+\sum_{p=i+2}^n\hM_{\emptyset,\lbrace i,p\rbrace}\\
&~~~-\sum_{p=1}^{i-1}\hM_{\emptyset,\lbrace p,i+1\rbrace}-\sum_{p=i+2}^n\hM_{\emptyset,\lbrace i+1,p\rbrace}\\
                                  &=\sum_{p=1}^{i-1}5(n+6)+\sum_{p=i+2}^n5(n+6)\\
                                  &~~~-\sum_{p=1}^{i-1}5(n+6)-\sum_{p=i+2}^n5(n+6)\\
                                  &=0.
\end{align*}

\paragraph{\textbf{Coordinates $k$}}
\begin{itemize}
\item For $k$ a positive integer such that $k<i$,
We have 
\begin{align*}
\big(M\cdot v_i\big)_k&=\sum_{p=1}^{i-1}\hM_{k,\lbrace p,i\rbrace}+\sum_{p=i+2}^n\hM_{k,\lbrace i,p\rbrace}-\sum_{p=1}^{i-1}\hM_{k,\lbrace p,i+1\rbrace}\\
                                  &~~-\sum_{p=i+2}^n\hM_{k,\lbrace i+1,p\rbrace}\\
                                  &=\sum_{p=1,p\neq k}^{i-1}\hM_{k,\lbrace p,i\rbrace}+\hM_{k\{k,i\}}+\sum_{p=i+2}^n\hM_{k,\lbrace i,p\rbrace}\\
                                  &~~-\sum_{p=1,p\neq k}^{i-1}\hM_{k,\lbrace p,i+1\rbrace}-\hM_{k,\{k,i+1}-\sum_{p=i+2}^n\hM_{k,\lbrace i+1,p\rbrace}\\
                                  &=\sum_{p=1,p\neq k}^{i-1}7+3+\sum_{p=i+2}^n7-\sum_{p=1,p\neq k}^{i-1}7-3-\sum_{p=i+2}^n7\\
                                  &=0.
\end{align*}
\item Similarly, for $k$ such that $i+1<k$, we have
$$\big(M\cdot v_i\big)_k=0.$$

\item For $k=i$ and $k=i+1$, we calculate
\begin{align*}
\big(M\cdot v_i\big)_i&=\sum_{p=1}^{i-1}\hM_{i,\lbrace p,i\rbrace}+\sum_{p=i+2}^n\hM_{i,\lbrace i,p\rbrace}\\
&~~~-\sum_{p=1}^{i-1}\hM_{i,\lbrace p,i+1\rbrace}-\sum_{p=i+2}^n\hM_{i,\lbrace i+1,p\rbrace}\\
                  &=\sum_{p=1}^{i-1}35+\sum_{p=i+2}^n3-\sum_{p=1}^{i-1}7-\sum_{p=i+2}^n7\\
                  &=35(i-1)+3(n-i-1)\\
                  &~~~-7(i-1)-7(n-i-1)\\
                  &=32i-24-4n
\end{align*}
and
\begin{align*}
\big(M\cdot v_i\big)_{i+1}&=\sum_{p=1}^{i-1}\hM_{i+1,\lbrace p,i\rbrace}+\sum_{p=i+2}^n\hM_{i+1,\lbrace i,p\rbrace}\\
                     &~~-\sum_{p=1}^{i-1}\hM_{i+1,\lbrace p,i+1\rbrace}-\sum_{p=i+2}^n\hM_{i+1,\lbrace i+1,p\rbrace}\\
                  &=\sum_{p=1}^{i-1}7+\sum_{p=i+2}^n7-\sum_{p=1}^{i-1}35-\sum_{p=i+2}^n3\\
                  &=7(i-1)+7(n-i-1)\\
                  &~~-35(i-1)-3(n-i-1)\\
                  &=-32i+24+4n.
\end{align*}

\end{itemize}

\paragraph{\textbf{Coordinates $\lbrace k,l\rbrace$}}
\begin{itemize}
\item For $k$ and $l$ distinct integers such that $1\leq k<l\leq i $, then 

\begin{align*}
\big(M\cdot v_i\big)_{\{ k,l\}}&=\sum_{p=1}^{i-1}\hM_{\{k,l\},\{p,i\}}+\sum_{p=i+2}^n\hM_{\{k,l\},\{i+1,p\}}\\
                               &~~-\sum_{p=1}^{i-1}\hM_{\{k,l\},\{ p,i+1\}}-\sum_{p=i+2}^n\hM_{\{k,l\},\{ i+1,p\}}\\
                               &=\sum_{p=1,p\neq k,l}^{i-1}\hM_{\{k,l\},\{p,i\}}+\hM_{\{k,l\},\{k,i\}}+\hM_{\{k,l\},\{l,i\}}\\
                               &~~+\sum_{p=i+2,p\neq k,l}^n\hM_{\{k,l\},\{i+1,p\}}-\sum_{p=1,p\neq k,l}^{i-1}\hM_{\{k,l\},\{ p,i+1\}}\\
                               &~~-\hM_{\{k,l\},\{k,i+1\}}-\hM_{\{k,l\},\{l,i+1\}}\\
                               &~~-\sum_{p=i+2,p\neq k,l}^n\hM_{\{k,l\},\{ i+1,p\}}\\
                               &=\sum_{p=1,p\neq k,l}^{i-1}5+3+3+\sum_{p=i+2,p\neq k,l}^n5\\
                               &~~-\sum_{p=1,p\neq k,l}^{i-1}5-3-3-\sum_{p=i+2,p\neq k,l}^n5\\                               
                               &=0.
\end{align*}

\item In the same way, in all cases with $k,l\neq i,i+1$, we have
$$\big(M\cdot v_i\big)_{\{k,l\}}=0.$$

\item We have

\begin{align*}
\big(M\cdot v_i\big)_{\lbrace i,i+1\rbrace}&=\sum_{p=1}^{i-1}\hM_{\{i,i+1\},\{p,i\}}+\sum_{p=i+2}^n\hM_{\{i,i+1\},\{i+1,p\}}\\
                                           &~~-\sum_{p=1}^{i-1}\hM_{\{i,i+1\},\{ p,i+1\}}-\sum_{p=i+2}^n\hM_{\{i,i+1\},\{ i+1,p\}}\\
                                           &=\sum_{p=1}^{i-1}15+\sum_{p=i+2}^n15-\sum_{p=1}^{i-1}15-\sum_{p=i+2}^n15\\
                                           &=0.
\end{align*} 
\item For $k<i$, we have

\begin{align*} 
\big(M\cdot v_i\big)_{\lbrace k,i\rbrace}=&\sum_{p=1}^{i-1}\hM_{\lbrace k,i\rbrace,\lbrace p,i\rbrace}+\sum_{p=i+2}^n\hM_{\lbrace k,i\rbrace,\lbrace i,p\rbrace}\\
                                      &-\sum_{p=1}^{i-1}\hM_{\lbrace k,i\rbrace,\lbrace                        p,i+1\rbrace}-\sum_{p=i+2}^n\hM_{\lbrace k,i\rbrace,\lbrace i+1,p\rbrace}\\
                                     =&\sum_{p=1,p\neq k}^{i-1}15+9+\sum_{p=i+2}^n3\\
                                     &-\sum_{p=1,p\neq k}^{i-1}5-3-\sum_{p=i+2}^n5\\
                                     =&15(i-2)+9+3(n-i-1)\\
                                     &-5(i-2)-3-5(n-i-1)\\
                                     =&12i-12-2n.
\end{align*}

Similarly, for $k<i$ and $l>i+1$, we have

\begin{align*}
\big(\hM\cdot v_i\big)_{\lbrace k,i+1\rbrace}&=-12i+12+2n\\
\big(\hM\cdot v_i\big)_{\lbrace i,l\rbrace}&=12i-12-2n\\
\big(\hM\cdot v_i\big)_{\lbrace i+1,l\rbrace}&=-12i+12+2n.
\end{align*}

\end{itemize}

Gathering these results together we have
$$M\cdot v_i=(32i-24-4n)u_i+(12i-12-2n)v_i.$$

Thus the vectors $u_i$ and $v_i$ form a stable plane for $M$
\end{proof}

We can form a base of $E$ containing the pairs $(u_i,v_i)$ for $1\leq i\leq n-1$, the matrix $M$ in this this base contains blocks as

$$\begin{pmatrix}
28 & 32i-24-4n \\
10 & 12i-12-2n
\end{pmatrix}$$
corresponding to the $i$th stable plane. The determinant of such a matrix is 
$$-16(n-i+6).$$

	This determinant vanishes for $i=n+6$, which never happens since $i<n$.

\subsubsection{Eigenvectors of \^{M}}
Further, for $1 \leq i < j < k < l \leq n$, introduce two vectors $w_{ijkl}$ and $t_{ijkl}$ in~$E$. These vectors have the following nonzero coordinates:
\begin{align*}
w_{ijkl} & : (\gamma_{ik}= 1, \gamma_{jl} = 1,  \gamma_{il} = -1, \gamma_{jk} = -1),\\
t_{ijkl} & : (\beta_j = 2, \beta_k = -2, \gamma_{ij} = -3, \gamma_{ik} = 3, \gamma_{jl}=-5, \gamma_{kl} = 5).
\end{align*}
All the coordinates that are not listed are equal to~0.

\begin{center}
\includegraphics[scale=0.4]{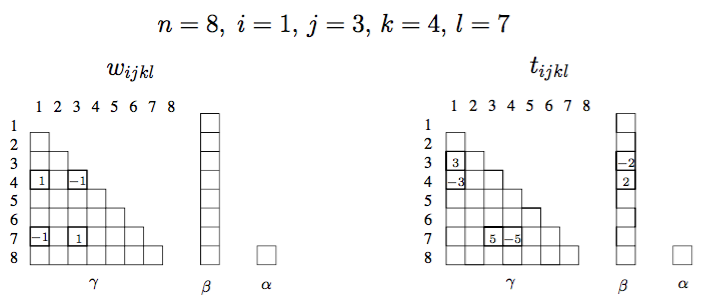}
\end{center}

\begin{prop}
The vectors $w_{ijkl}$ and $t_{ijkl}$ are eigenvectors of~$M$ with eigenvalue~$-4$ for any $i,j,k,l$.
\end{prop}

\begin{proof} Let's calculate the products $M\cdot w_{ijkl}$ and $M\cdot t_{ijkl}$.
\begin{itemize} 
\item The first coordinate of $M\cdot w_{ijkl}$ is
\begin{align*}
\big(\hM\cdot w_{ijkl}\big)_{\emptyset}&=\hM_{\emptyset,\lbrace i,k\rbrace}-\hM_{\emptyset,\lbrace i,l\rbrace}-\hM_{\emptyset,\lbrace j,k\rbrace}+\hM_{\emptyset,\lbrace j,l\rbrace}\\
&=5(n+6)-5(n+6)-5(n+6)+5(n+6)\\
&=0.
\end{align*}
\item For $p\neq i,j,k,l$, we have
\begin{align*}
\big(\hM\cdot w_{ijkl}\big)_p&=\hM_{p,\lbrace i,k\rbrace}-\hM_{p,\lbrace i,l\rbrace}-\hM_{p,\lbrace j,k\rbrace}+\hM_{p,\lbrace j,l\rbrace}\\
&=7-7-7+7\\
&=0.
\end{align*}

\item If $p=i$ or $p=j$, we have
\begin{align*}
\big(\hM\cdot w_{ijkl}\big)_i&=\hM_{i,\lbrace i,k\rbrace}-\hM_{i,\lbrace i,l\rbrace}-\hM_{i,\lbrace j,k\rbrace}+\hM_{i,\lbrace j,l\rbrace}\\
&=3-3-7+7\\
&=0
\end{align*}
and
\begin{align*}
\big(\hM\cdot w_{ijkl}\big)_j&=\hM_{j,\lbrace i,k\rbrace}-\hM_{j,\lbrace i,l\rbrace}-\hM_{j,\lbrace j,k\rbrace}+\hM_{j,\lbrace j,l\rbrace}\\
&=7-7-3+3\\
&=0.
\end{align*}

\item If $p=k$ or $p=l$, we have
\begin{align*}
\big(\hM\cdot w_{ijkl}\big)_j&=\hM_{k,\lbrace i,k\rbrace}-\hM_{k,\lbrace i,l\rbrace}-\hM_{k,\lbrace j,k\rbrace}+\hM_{k,\lbrace j,l\rbrace}\\
&=35-7-35+7\\
&=0
\end{align*}
and
\begin{align*}
\big(\hM\cdot w_{ijkl}\big)_l&=\hM_{l,\lbrace i,k\rbrace}-\hM_{l,\lbrace i,l\rbrace}-\hM_{l,\lbrace j,k\rbrace}+\hM_{l,\lbrace j,l\rbrace}\\
&=7-35-7+35\\
&=0.
\end{align*}

\item The coordinates $\lbrace p,q\rbrace$ with $p,q\neq i,j,k,l$ are
\begin{align*}
\big(\hM\cdot w_{ijkl}\big)_j&=\hM_{\lbrace p,q\rbrace,\lbrace i,k\rbrace}-\hM_{\lbrace p,q\rbrace,\lbrace i,l\rbrace}\\
&~~-\hM_{\lbrace p,q\rbrace,\lbrace j,k\rbrace}+\hM_{\lbrace p,q\rbrace,\lbrace j,l\rbrace}\\
&=5-5-5+5\\
&=0.
\end{align*}

\item The coordinates $\lbrace i,q\rbrace$ whith $q\neq j,k,l$ and $\lbrace j,q\rbrace$ with \linebreak $q\neq k,l$ are
\begin{align*}
\big(\hM\cdot w_{ijkl}\big)_j&=\hM_{\lbrace i,q\rbrace,\lbrace i,k\rbrace}-\hM_{\lbrace i,q\rbrace,\lbrace i,l\rbrace}\\
&~~-\hM_{\lbrace i,q\rbrace,\lbrace j,k\rbrace}+\hM_{\lbrace i,q\rbrace,\lbrace j,l\rbrace}\\
&=3-3-5+5\\
&=0
\end{align*}
and
\begin{align*}
\big(\hM\cdot w_{ijkl}\big)_j&=\hM_{\lbrace j,q\rbrace,\lbrace i,k\rbrace}-\hM_{\lbrace j,q\rbrace,\lbrace i,l\rbrace}\\
&~~-\hM_{\lbrace j,q\rbrace,\lbrace j,k\rbrace}+\hM_{\lbrace j,q\rbrace,\lbrace j,l\rbrace}\\
&=5-5-3+3\\
&=0.
\end{align*}

\item We have
\begin{align*}
\big(\hM\cdot w_{ijkl}\big)_{\lbrace i,j\rbrace}&=\hM_{\lbrace i,j\rbrace,\lbrace i,k\rbrace}-\hM_{\lbrace i,j\rbrace,\lbrace i,l\rbrace}\\
&~~-\hM_{\lbrace i,j\rbrace,\lbrace j,k\rbrace}+\hM_{\lbrace i,j\rbrace,\lbrace j,l\rbrace}\\
&=3-3-15+15\\
&=0,
\end{align*}
and
\begin{align*}
\big(\hM\cdot w_{ijkl}\big)_{\lbrace k,l\rbrace}&=\hM_{\lbrace k,l\rbrace,\lbrace i,k\rbrace}-\hM_{\lbrace k,l\rbrace,\lbrace i,l\rbrace}\\
&~~-\hM_{\lbrace k,l\rbrace,\lbrace j,k\rbrace}+\hM_{\lbrace k,l\rbrace,\lbrace j,l\rbrace}\\
&=3-15-3+15\\
&=0.
\end{align*}

\item We have
\begin{align*}
\big(\hM\cdot w_{ijkl}\big)_{\lbrace i,k\rbrace}&=\hM_{\lbrace i,k\rbrace,\lbrace i,k\rbrace}-\hM_{\lbrace i,k\rbrace,\lbrace i,l\rbrace}\\
&~~-\hM_{\lbrace i,k\rbrace,\lbrace j,k\rbrace}+\hM_{\lbrace i,k\rbrace,\lbrace j,l\rbrace}\\
&=9-3-15+5\\
&=-4.
\end{align*}

Similarly we calculate

\begin{align*}
\big(\hM\cdot w_{ijkl}\big)_{\lbrace i,l\rbrace}&=4\\
\big(\hM\cdot w_{ijkl}\big)_{\lbrace j,k\rbrace}&=4\\
\big(\hM\cdot w_{ijkl}\big)_{\lbrace j,l\rbrace}&=-4
\end{align*}

\end{itemize}

Hence $w_{ijkl}$ is an eigenvector of $\hM$ for the eigenvalue $-4$.

Let's calculate $\hM.t_{ijkl}$. The following coordinates are easy to calculate
\begin{itemize}
\item For $p\neq j,k$, $\big(\hM.t\big)_p=0$.
\item For $p,q\neq i,j,k,l$,
 \begin{align*}
 \big(\hM.t_{ijkl}\big)_{\lbrace p,q\rbrace}&=\big(\hM\cdot t_{ijkl}\big)_{\lbrace i,q\rbrace}=\big(\hM\cdot t_{ijkl}\big)_{\lbrace p,i\rbrace}\\
                                              &=\big(\hM\cdot t_{ijkl}\big)_{\lbrace k,q\rbrace}=\big(\hM\cdot t_{ijkl}\big)_{\lbrace p,k\rbrace}\\
                                              &=\big(\hM\cdot t_{ijkl}\big)_{\lbrace j,k\rbrace}=\big(\hM\cdot t_{ijkl}\big)_{\lbrace i,l\rbrace}\\
                                              &=\big(\hM\cdot t_{ijkl}\big)_{\lbrace p,j\rbrace}=\big(\hM\cdot t_{ijkl}\big)_{\lbrace j,q\rbrace}\\
                                              &=\big(\hM\cdot t_{ijkl}\big)_{\lbrace p,k\rbrace}=\big(\hM\cdot t_{ijkl}\big)_{\lbrace k,q\rbrace}\\
                                              &=0
\end{align*} 
\end{itemize}

We have
\begin{align*}
\big(\hM\cdot t_{ijkl}\big)_j=&-\big(\hM\cdot t_{ijkl}\big)_k\\
                          =&2\hM_{j,j}-2\hM_{j,k}-3\hM_{j,\lbrace i,j\rbrace}+3\hM_{j,\lbrace i,k\rbrace}\\
                           &-5\hM_{j,\lbrace j,l\rbrace}+5\hM_{j,\lbrace k,l\rbrace}\\
                          =&2\cdot 35-2\cdot 7-3\cdot 35+3\cdot 7-5\cdot 3+5\cdot 7\\
                          =&-8,
\end{align*}

\begin{align*}
\big(\hM\cdot t_{ijkl}\big)_{\lbrace i,j\rbrace}=&-\big(\hM\cdot t_{ijkl}\big)_{\lbrace i,k\rbrace}\\
                                             =&2\hM_{\lbrace i,j\rbrace,j}-2\hM_{\lbrace i,j\rbrace,k}-3\hM_{\lbrace i,j\rbrace,\lbrace i,j\rbrace}+3\hM_{\lbrace i,j\rbrace,\lbrace i,k\rbrace}\\
                                             &-5\hM_{\lbrace i,j\rbrace,\lbrace j,l\rbrace}+5\hM_{\lbrace i,j\rbrace,\lbrace k,l\rbrace}\\
                                             =&2\cdot 15-2\cdot 5-3\cdot 9+3\cdot 3-5\cdot 3+5\cdot 5\\
                                             =&12
\end{align*}

and

\begin{align*}
\big(\hM\cdot t_{ijkl}\big)_{\lbrace j,l\rbrace}=&-\big(\hM\cdot t_{ijkl}\big)_{\lbrace k,l\rbrace}\\
                                             =&2\hM_{\lbrace j,l\rbrace,j}-2\hM_{\lbrace j,l\rbrace,k}-3\hM_{\lbrace j,l\rbrace,\lbrace i,j\rbrace}+3\hM_{\lbrace j,l\rbrace,\lbrace i,k\rbrace}\\
                                             &-5\hM_{\lbrace j,l\rbrace,\lbrace j,l\rbrace}+5\hM_{\lbrace j,l\rbrace,\lbrace k,l\rbrace}\\
                                             =&2\cdot 15-2\cdot 5-3\cdot 15+3\cdot 5-5\cdot 9+5\cdot 15\\
                                             =&20
\end{align*}

Thus we have $\big(\hM\cdot t_{ijkl}\big)=-4t_{ijkl}$.
\end{proof}

\subsubsection{The exceptional case}

When $n$ is congruent to $2$ modulo $8$, one of the stable planes contains an eigenvector for the eigenvalue $-4$. In this cases, that we called exceptional, we need to find a new vector in order to determine the rank of $\hM$. In this purpose, we introduce the following vector
$$
z: (\gamma_{3m+1,n} = 1, \gamma_{3m+2,n} = -1).
$$
The other coordinates of~$z$ vanish.

\begin{prop}
If $n = 8p+2$,  the invariant plane spanned by $u_{3p+1}$ and $v_{3p+1}$ contains an eigenvector with eigenvalue~$-4$. Moreover, the vector $\displaystyle{z +\frac{m}{5m+7}u_{3m+1}}$ is annihilated by $(M+4)^2$, but not by $M+4$. 
\end{prop}

\begin{proof}
When $n$ is congruent to $2$ modulo $8$, in other terms, when there exists $m\in\N$ such that $n=8m+2$, the stable plan spanned by $u_{3m+1}$ and $v_{3m+1}$ is associated to the matrix
$$
\begin{pmatrix}
28 & 32(3m+1)-24-4(8m+2) \\
10 & 12(3m+1)-12-2(8m+2)
\end{pmatrix}
=
\begin{pmatrix}
28 & 64m \\
10 & 20m-4
\end{pmatrix}.
$$
Its characteristic polynomial is
$$\lambda^2-(20m+24)\lambda-80m-112$$
which factorizes into
$$(\lambda+4)(\lambda-20m-28).$$

Hence we have two eigenvectors.
\paragraph{\textbf{For the eigenvalue $-4$}} The system
$$\left\{
\begin{array}{cl}
28x-64my=-4x\\
10x(20m-4)y=-4y
    \end{array}
\right. \Leftrightarrow x+2my=0
$$
shows that $\mathcal{U}:=2mu_{3m+1}-v_{3m+1}$ is an eigenvector for the eigenvalue~$-4$.

~\\

If we write $\mathcal{V}$ an eigenvector for the eigenvalue $20m+28$, in the base $(\mathcal{U},\mathcal{V})$, the block corresponding to the $(3m+1)$st stable plan is written 
$$
\begin{pmatrix}
-4 & 0 \\
0  & 20m+28
\end{pmatrix}
$$

\paragraph{Vector of \textbf{$Ker(M+4)^2\backslash Ker(M+4)$}}
Now we consider the vector $z$.\\
 We have

\begin{itemize}
\item \begin{align*}
\big(\hM\cdot z\big)_{3m+1}&=-\big(\hM.z\big)_{3m+2}\\
                     &=\hM_{3m+1,\{ 3m+1,n\}}-\hM_{3m+1,\{ 3m+2,n\}}\\
                     &=3-7=-4.
\end{align*}

\item \begin{align*}
\big(\hM\cdot z\big)_{\{ p,3m+1\}}&=\hM_{\{ p,3m+1\},\{ 3m+1,n\}}-\hM_{\{ p,3m+1\},\{ 3m+2,n\}}\\
   &=3-5=-2.
\end{align*}

\item \begin{align*}
\big(\hM\cdot z\big)_{\{ 3m+1,p\}}&=\hM_{\{ 3m+1,p\},\{ 3m+1,n\}}-\hM_{\{ 3m+1,p\},\{ 3m+2,n\}}\\
   &=3-5=-2.
\end{align*}

\item \begin{align*}
\big(\hM\cdot z\big)_{\{ 3m+1,m+2\}}&=\hM_{\{ 3m+1,3m+2\},\{ 3m+1,n\}}\\
&~~-\hM_{\{ 3m+1,m+2\},\{ 3m+2,n\}}\\
   &=3-3=0.
\end{align*}

\item \begin{align*}
\big(\hM\cdot z\big)_{\{ 3m+1,n\}}&=\hM_{\{ 3m+1,n\},\{ 3m+1,n\}}-\hM_{\{ 3m+1,n\},\{ 3m+2,n\}}\\
   &=9-15=-6.
\end{align*}
\end{itemize}

Then $(\hM+4)\cdot z=-4u_{3m+1}-2v_{3m+1}$.

Now let's write $\displaystyle{\delta:=\frac{m}{5m+7}}$ and $\mathcal{Z}:=z+\delta u_{3m+1}$. We have

\begin{align*}
(\hM+4)\cdot \mathcal{Z}&=(\hM+4)\cdot z+\delta(\hM+4)\cdot u_{3m+1}\\
                   &=-4u-{3m+1}-2v_{3m+1}+\delta(28u_{3m+1}+10v_{3m+1})+4u_{3m+1}\\
                   &=28\delta u_{3m+1}+(10\delta-2)v_{3m+1}\\
                   &=\frac{-14}{5m+7}\mathcal{U}.
\end{align*}
 Since $\mathcal{U}$ is an eigenvector associated to the eigenvalue $-4$, we have
 $$(\hM+4)^2\cdot\mathcal{Z}=0.$$

\end{proof}

\subsubsection{The span of vectors u, v, w, t, Z} 

\begin{prop}
The vectors $u_i$, $v_i$, $w_{ijkl}$, $t_{ijkl}$, together with~$z$ in the case $n = 8p+2$, span the codimension~3 subspace of~$E$ given by the equations 
$$
\alpha=0,  \quad \sum_{i=1}^n \beta_i =0, \quad
\sum_{i<j} \gamma_{ij}=0.
$$
\end{prop}

\paragraph{Proof.} First of all, it is easy to see that all vectors $u,v,w,t,z$ satisfy the three linear equations given above.

Now consider a vector $s$ in $E$ satisfying the three equations. We will show that we can make it vanish by adding an appropriate linear combination of vectors $u_i$, $v_i$, $w_{ijkl}$, $t_{ijkl}$, and $z$ if $n = 8p+2$.

For $1 \leq i \leq n-2$ let
$$
\tv_i = v_i - \frac13 \left( t_{1,i,i+1,n} + t_{2,i,i+1,n} + \dots + t_{i-1,i,i+1,n} \right).
$$

\begin{center}
\includegraphics[scale=0.38]{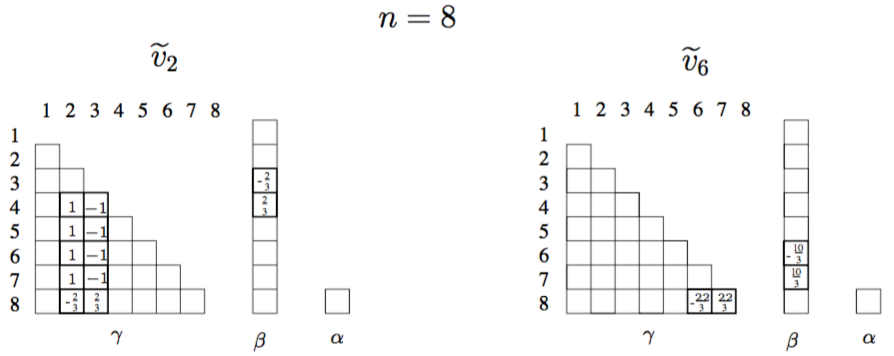}
\end{center}

This vector has the property that its only nonzero $\gamma$-coordinates have the form $\gamma_{ij}$ or $\gamma_{i+1,j}$. Moreover, the coordinates with first index~$i$ add up to $(n-1-i)-\frac53(i-1) = \frac{3n+2-8i}3$.
Similary, its coordinates with first index $i+1$ add up to the opposite number $-\frac{3n+2-8i}3$.

Now we perform the {\bf first step} that consists in annihilating the $\gamma$-coordinates of~$s$. This is done using the following sequence of operations.

As a preparation we add to $s$ a multiple of $v_{n-1}$ so as to annihilate the sum $\sum_{i=1}^{n-1} \gamma_{in}$. We will not use the vector $v_{n-1}$ again in the following operations, and one can easily check that in all other vectors $u,v,w,t,z$ the sum of coordinates $\gamma_{in}$ is equal to~0. Thus none of the following operations will change this sum, and it will remain equal to~0 all the time. Now we can start killing the coordinates in earnest.

First we add a multiple of $t_{1234}$ to annihilate $\gamma_{12}$. Then we add a multiple of $v_1 = \tv_1$ so as to annihilate the sum $\sum_j \gamma_{1j}$. Note that this does not change $\gamma_{12}$. Now we use the vectors $w_{12kl}$ to annihilate all the coordinates $\gamma_{1j}$. Note that this does not change either $\gamma_{12}$ or $\sum_j \gamma_{1j}$.

We have now achieved the vanishing of all coordinates $\gamma_{1j}$ . The next step is to kill the coordinates $\gamma_{2j}$, and then we continue to increase the first index of the $\gamma$-coordinates one by one. Assume that we have already achieved the vanishing of all $\gamma$-coordinates with first index less than~$i$ and let us do the coordinates $\gamma_{ij}$ for $j$ from $i+1$ to~$n$. We assume that $i \leq n-3$. 

First we use the vector $t_{i,i+1,i+2,i+3}$ to kill the coordinate $\gamma_{i,i+1}$. Second, we use the vector $\tv_i$ to annihilate the sum
$\sum_{j=i+1}^n \gamma_{ij}$. This is always possible unless $3n+2-8i = 0$. In this case we have $n = 8p+2$, $i=3p+1$, so we can use the vector~$z$ instead. Third, we use the vectors $w_{i,i+1,k,l}$ to kill all coordinates $\gamma_{ij}$ for our given value of~$i$. Note that all these operations do not change the coordinates $\gamma_{i',j}$ with $i' <i$. 

Once we are finished with $i = n-3$ we are left with only three possibly nonzero $\gamma$-coordinates: $\gamma_{n-2,n-1}$, $\gamma_{n-2,n}$, and $\gamma_{n-1,n}$. Recall that as a preliminary step we have achieved the vanishing $\sum \gamma_{in}=0$; now this condition reads $\gamma_{n-2,n} = - \gamma_{n-1,n}$. Thus we can use the vector $\tv_{n-2}$ to kill both $\gamma_{n-2,n}$ and $\gamma_{n-1,n}$. Now the remaining coordinate $\gamma_{n-2,n-1}$ is automatically equal to 0, because the sum of all $\gamma$-coordinates of $s$ was 0 from the start. 

After getting rid of the $\gamma$-coordinates, the {\bf second step} consists in eliminating the $\beta$-coordinates without changing the $\gamma$-coordinates. This step is much simpler: we just use the vectors $u_i$.
\qed

\subsubsection{Complement subspace}

 We write $F$ the span of $u, v, w, t, \cZ$. If we write 
 $$ \left( \begin{array}{c}
\alpha \\
\beta_1 \\
\vdots \\
\beta_n \\
\gamma_{1,2}\\
\vdots \\
\gamma_{n-1,n}
\end{array}\right)$$
any vector of $E$, $F$ corresponds to the subspace of vectors satisfying the equations
\begin{align*}
\alpha&=0\\
\sum_{i=1}^n\beta_i&=0\\
\sum_{i<j}\gamma_{i,j}&=0.
\end{align*}
We look now for a complement of $F$ in $E$. Hence we introduce the vectors $a, b$ and $c$ of $E$ with the following nonzero coordinates
\begin{align*}
a &: (\alpha =1),\\
b &: (\beta_1=1),\\
c &: (\gamma_{1,2}=1).
\end{align*}

We calculate the image of $a$. We have
$$\big(\hM\cdot a\big)_{\emptyset}=\hM_{\emptyset,\emptyset}=5(n+6).$$
For $1\leq i\leq n$, we have
$$\big(\hM\cdot a\big)_i=\hM_{i,\emptyset}=7$$
and for $1\leq i<j\leq n$, we have
$$\big(\hM\cdot a\big)_{\{ i,j\}}=\hM_{\{i,j\},\emptyset}=5.$$

Now we calculate the image of $b$. We have
$$\big(\hM\cdot b\big)_{\emptyset}=\hM_{\emptyset,1}=5n+34$$
and
$$\big(\hM\cdot b\big)_1=\hM_{1,1}=35$$

For $2\leq i\leq n$, we have
$$\big(\hM\cdot b\big)_i=\hM_{i,1}=7$$
and 
$$\big(\hM\cdot b\big)_{\{1,i\}}=\hM_{\{ 1,i\},1}=15.$$

For $2\leq i,j\leq n$, we have
$$\big(\hM\cdot b\big)_{\{i,j\}}=\hM_{\{ i,j\},1}=5.$$

Let's calculate the image of $c$. For the first coordinate we have
$$\big(\hM\cdot c\big)_{\{1,2\}}=\hM_{\emptyset,\{1,2\}}=5(n+6),$$
$$\big(\hM\cdot c\big)_1=\hM_{1,\{1,2\}}=3$$
and
$$\big(\hM\cdot c\big)_2=\hM_{2,\{1,2\}}=35.$$
For $3\leq i\leq n$, we have
$$\big(\hM\cdot c\big)_i=\hM_{i,\{ 1,2\}}=7.$$
We have
$$\big(\hM\cdot c\big)_{\{1,2\}}=\hM_{\{1,2\},\{1,2\}}=9.$$
For $2\leq i\leq n$, we have
$$\big(\hM\cdot c\big)_{\{1,i\}}=\hM_{\{1,i\},\{1,2\}}=3$$
and for $3\leq i\leq n$, we have
$$\big(\hM\cdot c\big)_{\{2,i\}}=\hM_{\{2,i\},\{1,2\}}=15.$$
For $3\leq i<j\leq n$, we have
$$\big(\hM\cdot c\big)_{\{i,j\}}=\hM_{\{i,j\},\{1,2\}}=5.$$

Modulo the subspace $F$, the images of these vectors only depend on the first coordinate, the sum of the $n$ following coordinates and the sum of the others coordinates, hence on $3$ numbers. We can write that, modulo $F$,

$$\big(\hM\cdot a\big)\sim \left( \begin{array}{c}
\big(\hM\cdot a\big)_{\emptyset} \\
\sum_{i=1}^n\big(\hM\cdot a\big)_i \\
0\\
\vdots \\
0\\
\sum_{i<j}\big(\hM\cdot a\big)_{\{i,j\}}\\
\vdots \\
0
\end{array}\right), 
\big(\hM\cdot b\big)\sim  \left( \begin{array}{c}
\big(\hM\cdot b\big)_{\emptyset} \\
\sum_{i=1}^n\big(\hM\cdot b\big)_i\\
0\\
\vdots \\
0\\
\sum_{i<j}\big(\hM\cdot b\big)_{\{i,j\}}\\
0\\
\vdots \\
0
\end{array}\right)$$

$$
\mbox{ and }
\big(\hM\cdot c\big)\sim 
 \left( \begin{array}{c}
\big(\hM\cdot c\big)_{\emptyset} \\
\sum_{i=1}^n\big(\hM\cdot c\big)_i \\
0\\
\vdots \\
0\\
\sum_{i<j}\big(\hM\cdot c\big)_{\{i,j\}}\\
0\\
\vdots \\
0
\end{array}\right).
$$

Let's calculate these coordinates, we have

\begin{align*}
\big(\hM\cdot a\big)&=\hM_{\emptyset,\emptyset}a+\sum_{i=1}^n\hM_{i,\emptyset}b+\sum_{i<j}\hM_{\{i,j\},\emptyset}c\\
          &=5(n+6)a+7nb+5\binom{n}{2}c.\\
\end{align*}

\begin{align*}
\big(\hM\cdot b\big)&=\hM_{\emptyset,1}a+\sum_{i=1}^n\hM_{i,1}b+\sum_{i<j}^n\hM_{\{i,j\},1}c\\
  &=(5n+34)a+\big(35+7(n-1)\big)b\\
  &~~+\big(15(n-1)+5((\binom{n}{2}-(n-1))\big)c\\
  &=(5n+34)a+(7n+28)b+\frac{5}{2}(n-1)(n+4)c.
\end{align*}

\begin{align*}
\big(\hM\cdot c\big)=&\hM_{\emptyset,\{1,2\}}a+\sum_{i=1}^n\hM_{i,\{1,2\}}b+\sum_{i<j}\hM_{\{i,j\},\{1,2\}}c\\
   =&5(n+6)a+\big(3+35+7(n-2)\big)b\\
    &+\Big(9+3(n-2)+15(n-2)\\
    &~~~~~+5\big(\binom{n}{2}-(n-1)-(n-2)\big)\Big)c\\
   =&5(n+6)a+(7n+24)b+\big(\frac{5}{2}n^2+\frac{11}{2}n-12\big)c.
\end{align*}

Hence we have the following block
$$
\begin{pmatrix}
5(n+6) & 5n+34 & 5(n+6) \\
7n & 7n+28 & 7n+24 \\
\frac{5}{2}n(n-1) & \frac{5}{2}(n-1)(n+4) & \frac{5}{2}n^2+\frac{11}{2}n-12
\end{pmatrix}.
$$
Its determinant is $-32(n+6)(2n+15)$ which never vanishes for $n\in\mathbb{N}$.

\subsection{Conclusion}

We have rewritten $\hM$ as a block upper triangular matrix whose block have been described in the previous sections. The determinant of $\hM$ is now easy to calculate through this new matrix. This last matrix is composed by $n-1$ blocks of size $2$, a block of size $3$ and blocks of size $-4$. The latter type of block appears $\frac{n(n-1)}{2}+n+1-\big[2(n-1)+3\big]=\frac{n(n-3)}{2}$~~times. Then we have
$$\det\hM=\prod_{i=1}^{n-1}(-16(n-i+6))\cdot (-4)^{\frac{n(n-3)}{2}}\cdot(-32)(n+6)(2n+15)$$
which is better written
$$\det\hM=(-1)^{\frac{n(n-1)}{2}}2^{n^2+n+1}(2n+15)\frac{(n+6)!}{6!}.$$
In this form we see clearly that $\det\hM$ never vanishes, proving that the classes $\kappa_{1,1},\psi_i^2,\psi_i\psi_j$ are independent. Since we know that the other classes lie in the span of these classes, they generate the degree $2$ group of $R^*(\cM_{4,n})$. Hence we have
$$R^2(\cM_{4,n})\simeq <\psi_1^2,...,\psi_n^2,\kappa_{1,1}>.$$

~\\

Thus we obtained the complete description of the ring.

\newpage
\bibliographystyle{plain}
\bibliography{biblio}

\end{document}